\documentclass[10pt,reqno]{amsart}
\pdfoutput=1

\usepackage{adjustbox,array,bigints,cancel,color,comment,extpfeil,mathdots,mathrsfs,mathtools,MnSymbol,multicol,scalerel,setspace,tcolorbox,tikz,tikz-cd,upgreek}
\usepackage[fontsize = 10pt]{fontsize}
\usepackage[left = 2.5cm, right = 2.5cm, top = 2.5cm, bottom = 2.5cm]{geometry}
\usepackage{hyperref}
\usetikzlibrary{patterns}
\numberwithin{equation}{section}

\theoremstyle{plain}
\newtheorem{Cl}[equation]{Claim}

\newtheorem{Lem}[equation]{Lemma}

\newtheorem{Prop}[equation]{Proposition}

\newtheorem{Thm}[equation]{Theorem}

\theoremstyle{definition}

\newtheorem{Defn}[equation]{Definition}

\theoremstyle{remark}

\theoremstyle{plain}
\newtheorem*{Cl*}{Claim}
\newtheorem*{Conj*}{Conjecture}
\newtheorem*{Lem*}{Lemma}
\newtheorem*{Prop*}{Proposition}
\newtheorem*{Q*}{Question}
\newtheorem*{Schol*}{Scholium}
\newtheorem*{SubCl*}{Subclaim}
\newtheorem*{Thm*}{Theorem}

\theoremstyle{definition}
\newtheorem*{Cond*}{Condition}
\newtheorem*{Cstr*}{Construction}
\newtheorem*{Defn*}{Definition}
\newtheorem*{Ex*}{Example}
\newtheorem*{Exs*}{Examples}
\newtheorem*{Md*}{Method}
\newtheorem*{Nt*}{Notation}
\newtheorem*{Pty*}{Property}

\theoremstyle{remark}

\newtheorem*{Rk*}{Remark}
\newtheorem*{Rks*}{Remarks}
\newtheorem*{A-d}{Aside}

\newcommand{\cag}{\begin{equation}\begin{gathered}}
\newcommand{\caag}{\end{gathered}\end{equation}}
\newcommand{\caw}{\begin{equation*}\begin{gathered}}
\newcommand{\caaw}{\end{gathered}\end{equation*}}
\newcommand{\e}{\begin{equation}\begin{aligned}}
\newcommand{\ee}{\end{aligned}\end{equation}}
\newcommand{\ew}{\begin{equation*}\begin{aligned}}
\newcommand{\eew}{\end{aligned}\end{equation*}}

\newcommand{\bcd}{\begin{tikzcd}}
\newcommand{\ecd}{\end{tikzcd}}
\newcommand{\bma}{\begin{matrix}}
\newcommand{\ema}{\end{matrix}}
\newcommand{\bpm}{\begin{pmatrix}}
\newcommand{\epm}{\end{pmatrix}}
\newcommand{\bvm}{\begin{vmatrix}}
\newcommand{\evm}{\end{vmatrix}}

\newcommand{\nts}{\begin{tcolorbox}}
\newcommand{\ntss}{\end{tcolorbox}}

\newcommand{\cref}[1]{Corollary \ref{#1}}

\newcommand{\eref}[1]{eqn.\hspace{0.6mm}(\ref{#1})}
\newcommand{\erefs}[1]{eqns.\hspace{0.6mm}(\ref{#1})}

\newcommand{\lref}[1]{Lemma \ref{#1}}

\newcommand{\pref}[1]{Proposition \ref{#1}}

\newcommand{\sref}[1]{\S\ref{#1}}

\newcommand{\tref}[1]{Theorem \ref{#1}}
\newcommand{\trefs}[1]{Theorems \ref{#1}}

\DeclareMathSizes{10}{10}{8}{7}

\newcommand{\mss}[1]{\mbox{\scriptsize \(#1\)}}

\newcommand{\mns}[1]{\mbox{\normalsize \(#1\)}}
\newcommand{\mla}[1]{\mbox{\large \(#1\)}}

\newcommand{\bb}[1]{\mathbb{#1}}
\newcommand{\cal}[1]{\mathscr{#1}}
\newcommand{\fr}[1]{\mathfrak{#1}}

\newcommand{\mc}[1]{\mathcal{#1}}

\newcommand{\et}{\hspace{3mm}\text{and}\hspace{3mm}}
\newcommand{\hs}[1]{\hspace{#1}}

\newcommand{\vs}[1]{\vspace{#1}}

\DeclareMathSymbol{\Alpha}{\mathalpha}{operators}{"41}
\DeclareMathSymbol{\Beta}{\mathalpha}{operators}{"42}
\DeclareMathSymbol{\Epsilon}{\mathalpha}{operators}{"45}
\DeclareMathSymbol{\Zeta}{\mathalpha}{operators}{"5A}
\DeclareMathSymbol{\Eta}{\mathalpha}{operators}{"48}
\DeclareMathSymbol{\Iota}{\mathalpha}{operators}{"49}
\DeclareMathSymbol{\Kappa}{\mathalpha}{operators}{"4B}
\DeclareMathSymbol{\Mu}{\mathalpha}{operators}{"4D}
\DeclareMathSymbol{\Nu}{\mathalpha}{operators}{"4E}
\DeclareMathSymbol{\Omicron}{\mathalpha}{operators}{"4F}
\DeclareMathSymbol{\Rho}{\mathalpha}{operators}{"50}
\DeclareMathSymbol{\Tau}{\mathalpha}{operators}{"54}
\DeclareMathSymbol{\Chi}{\mathalpha}{operators}{"58}
\DeclareMathSymbol{\omicron}{\mathord}{letters}{"6F}

\newcommand{\al}{\alpha}
\newcommand{\be}{\beta}
\newcommand{\ga}{\gamma}
\newcommand{\de}{\delta}
\newcommand{\ep}{\varepsilon}

\renewcommand{\th}{\theta}

\newcommand{\io}{\iota}

\newcommand{\la}{\lambda}
\newcommand{\vpi}{\varpi}

\newcommand{\si}{\sigma}

\newcommand{\ph}{\phi}
\newcommand{\vph}{\upvarphi}
\newcommand{\vps}{\uppsi}
\newcommand{\ch}{\chi}
\newcommand{\ps}{\psi}

\newcommand{\Ga}{\Gamma}
\newcommand{\De}{\Delta}

\newcommand{\Om}{\Omega}

\newcommand{\Aut}{\operatorname{Aut}}

\newcommand{\Ds}{\bigoplus}
\newcommand{\ds}{\oplus}

\newcommand{\lqt}[2]{\left.\raisebox{-1mm}{\(#2\)}\middle\backslash\raisebox{1mm}{\(#1\)}\right.}

\newcommand{\rqt}[2]{\left.\raisebox{1mm}{\(#1\)}\middle/\raisebox{-1mm}{\(#2\)}\right.}
\newcommand{\ts}{\otimes}

\newcommand{\x}{\times}

\newcommand{\del}{\partial}

\newcommand{\Hocl}[1]{\overset{\circ}{H^{#1}}_{\kern-1.9mm\cl}}

\newcommand{\lop}{\left\|\kern-1.30mm\left\|}
\newcommand{\op}{\|\kern-1.30mm\|}
\newcommand{\rop}{\right\|\kern-1.30mm\right\|}
\newcommand{\SI}{\operatorname{\cal{I}\kern-1.5pt nd}}

\newcommand{\C}{\Subset}
\newcommand{\cc}{\subseteq}

\newcommand{\es}{\emptyset}

\newcommand{\mt}{\mapsto}

\newcommand{\osr}{\backslash}
\newcommand{\oto}[1]{\xrightarrow{#1}}
\newcommand{\pc}{\subset}

\newcommand{\B}{\mathrm{B}}
\newcommand{\CCH}[1][]{\mathcal{H}_{4#1}}

\newcommand{\CH}[1][]{\mathcal{H}_{3#1}}

\newcommand{\cl}{\mathrm{closed}}

\newcommand{\dd}{\mathrm{d}}

\newcommand{\dR}[1]{H^{#1}_{\operatorname{dR}}}

\newcommand{\emb}{\hookrightarrow}

\newcommand{\hk}{\righthalfcup}

\newcommand{\Hs}{\raisebox{1pt}{\mss{\bigstar}}}
\newcommand{\itr}[1]{\overset{\circ}{#1}}

\newcommand{\M}{\mathrm{M}}

\newcommand{\SCCH}[1][]{\widetilde{\mathcal{H}}_{4#1}}
\newcommand{\SCH}[1][]{\widetilde{\mathcal{H}}_{3#1}}

\newcommand{\sph}{\widetilde{\phi}}
\newcommand{\sps}{\widetilde{\psi}}
\renewcommand{\ss}[2][{}]{\bigodot{\hspace{-1mm}}^{#2}_{#1}\hspace{0.6mm}}
\newcommand{\supp}{\operatorname{supp}}
\newcommand{\svph}{\widetilde{\upvarphi}}
\newcommand{\svps}{\widetilde{\uppsi}}
\newcommand{\T}{\mathrm{T}}
\newcommand{\tl}{{\mns{\sim}}}

\newcommand{\w}{\wedge}
\newcommand{\ww}[2][{}]{\bigwedge{\hspace{-1mm}}^{#2}_{\hspace{1mm}#1}\hspace{0.1mm}}

\newcommand{\0}{\infty}
\newcommand{\1}{\cdot}

\renewcommand{\ge}{\geqslant}
\newcommand{\gl}{\hspace{0.4mm}\raisebox{0.8mm}{\(>\)}\kern-1.8mm\raisebox{-0.8mm}{\(<\)}\hspace{0.4mm}}
\newcommand{\gle}{\hspace{0.4mm}\raisebox{1.2mm}{\(\ge\)}\kern-1.8mm\raisebox{-1.2mm}{\(\le\)}\hspace{0.4mm}}
\renewcommand{\le}{\leqslant}
\newcommand{\pt}{\bullet}

\newcommand{\diag}{\operatorname{diag}}
\newcommand{\g}{\(\mathrm{G}_2\)}
\newcommand{\Gg}{\mathrm{G}}
\newcommand{\GL}{\operatorname{GL}}

\newcommand{\sg}{\(\widetilde{\mathrm{G}}_2\)}

\newcommand{\SO}{\operatorname{SO}}

\newcommand{\Stab}{\operatorname{Stab}}

\renewcommand{\iff}{if and only if}

\newcommand{\wlg}{without loss of generality}

\newcommand{\wrt}{with respect to}

\newcommand{\ol}[1]{\overline{#1}}
\newcommand{\h}{\widehat}
\newcommand{\lt}{\left}
\newcommand{\m}{\middle}

\newcommand{\rt}{\right}

\newcommand{\tld}{\widetilde}

\title{Unboundedness above and below of the Donaldson--Hitchin functionals on \g- and \sg-forms}
\author{Laurence H. Mayther}

\begin{document}\fontsize{10pt}{12pt}\selectfont
\begin{abstract}
\footnotesize{This paper combines explicit local calculations with covering arguments to prove the unboundedness above and below (in a logarithmic sense) of the Donaldson--Hitchin functionals on \g\ 4-forms, \sg\ 3-forms and \sg\ 4-forms, over compact manifolds (or, more generally, orbifolds) with boundary.  In addition, the Donaldson--Hitchin functional on \g\ 3-forms over compact manifolds (or orbifolds) with boundary is shown to be unbounded below.  As scholia, the critical points of the functionals on \g\ 4-forms, \sg\ 3-forms and \sg\ 4-forms are shown to be saddles, and initial conditions of the Laplacian coflow which do not lead to convergent solutions are shown to be dense.}
\end{abstract}
\maketitle

\section{Introduction}

Let \(\M\) be an oriented 7-manifold.  A 3-form \(\ph\) on \(\M\) is termed a \g\ 3-form if, for each \(x \in \M\), there exists an orientation-preserving isomorphism \(\al:\T_x\M \to \bb{R}^7\) identifying \(\ph|_x\) with the 3-form:
\ew
\vph_0 = \th^{123} + \th^{145} + \th^{167} + \th^{246} - \th^{257} - \th^{347} - \th^{356} \in \ww{3}\lt(\bb{R}^7\rt)^*,
\eew
where \(\lt(\th^1,...,\th^7\rt)\) denotes the standard basis of \(\lt(\bb{R}^7\rt)^*\) and multi-index notation \(\th^{ij...k} = \th^i \w \th^j \w ... \w \th^k\) is used throughout this paper.  Since the stabiliser of \(\vph_0\) in \(\GL_+(7;\bb{R})\) is \g, a \g\ 3-form \(\ph\) on \(\M\) is equivalent to a \g-structure on \(\M\), i.e.\ a principal \g-subbundle of the frame bundle of \(\M\).  Since \(\Gg_2 \pc \SO(7)\) preserves the Euclidean inner product \(g_0\) and volume element \(vol_0\) on \(\bb{R}^7\), it follows that \(\ph\) induces a Riemannian metric \(g_\ph\) and volume form \(vol_\ph\) on \(\M\).  Say that \(\ph\) (or, more generally, the corresponding \g-structure on \(\M\)) is torsion-free if both \(\dd\ph = 0\) and \(\dd\Hs_\ph\ph = 0\), where \(\Hs_\ph\) is the Hodge star induced by \(g_\ph\) and \(vol_\ph\); in this case, by a well-known result of Fern\'{a}ndez--Gray \cite{RMwSGG2}, the holonomy group of the metric \(g_\ph\) is a subgroup of the exceptional holonomy group \g.

\g\ 3-forms are stable in the sense of Hitchin \cite{SF&SM} in that the \(\GL_+(7;\bb{R})\) orbit of \(\vph_0\) in \(\ww{3}\lt(\bb{R}^7\rt)^*\) is open.  Thus, following Hitchin, if \(\M\) is closed then given a closed \g\ 3-form \(\ph\) on \(\M\), write \([\ph] \in \dR{3}(\M)\) for the de Rham class defined by \(\ph\), define \([\ph]_+ = \lt\{ \ph' \in [\ph] ~\m|~ \ph' \text{ is a \g\ 3-form}\rt\}\) and define the Hitchin functional \(\CH: [\ph]_+ \to (0,\0)\) by:
\ew
\CH(\ph') = \bigintsss_\M vol_{\ph'}.
\eew
\(\CH\) has the important geometric property that its critical points are precisely those \(\ph' \in [\ph]_+\) which are torsion-free.

In the more general case where \(\M\) is compact with boundary, whilst the set \([\ph]_+\) and the functional \(\CH\) remain well defined, it is more natural (as observed by Donaldson in \cite{BVPiD74&3RtEH,AEBVPfG2S}) from the perspective of boundary value problems to restrict attention to those \(\ph' \in [\ph]_+\) which have `the same boundary value' as \(\ph\).  This leads naturally to a boundary version of the functional \(\CH\), which I shall term the Donaldson--Hitchin functional \(\CH^\del\), defined as follows.  Let \(\Om^p(\M; \del \M) = \lt\{\al \in \Om^p(\M) ~\m|~ \al|_x = 0 \text{ for all } x \in \del\M \rt\}\) and set:
\ew
[\ph]_+^\del = \lt\{ \ph' \in \ph + \dd\Om^2(\M;\del\M) ~\m|~ \ph' \text{ is a \g\ 3-form}\rt\}.
\eew
Then \(\CH^\del:[\ph]_+^\del \to (0,\0)\) is defined by:
\ew
\CH^\del(\ph') = \bigintsss_\M vol_{\ph'}.
\eew
Again, the functional \(\CH^\del\) has the important property that its critical points are precisely those \(\ph' \in [\ph]_+^\del\) which are torsion-free.  Moreover, in the case \(\del\M = \es\) one has \([\ph]_+^\del = [\ph]_+\) and thus \(\CH^\del\) is simply the usual Hitchin functional \(\CH\).

This paper considers two natural generalisations of the functional \(\CH^\del\).  Firstly, consider the forms:
\caw
\vps_0 = \th^{4567} + \th^{2367} + \th^{2345} + \th^{1357} - \th^{1346} - \th^{1256} - \th^{1247} \in\ww{4}(\bb{R}^7)^*;\\
\svph_0 = \th^{123} - \th^{145} - \th^{167} + \th^{246} - \th^{257} - \th^{347} - \th^{356} \in\ww{3}(\bb{R}^7)^*;\\
\svps_0 = \th^{4567} - \th^{2367} - \th^{2345} + \th^{1357} - \th^{1346} - \th^{1256} - \th^{1247} \in\ww{4}(\bb{R}^7)^*.
\caaw
Exterior forms on \(\M\) modelled on \(\vps_0\), \(\svph_0\) and \(\svps_0\) are termed \g\ 4-forms, \sg\ 3-forms and \sg\ 4-forms, respectively, and are equivalent to \g- and \sg-structures on \(\M\) as appropriate.  (Here, \sg\ denotes the centreless, doubly connected Lie group corresponding to the split real form of the exceptional Lie algebra \(\fr{g}_{2,\bb{C}}\).)  Such forms are, once again, stable and one can define corresponding Donaldson--Hitchin functionals in each case, here denoted \(\CCH^\del\), \(\SCH^\del\) and \(\SCCH^\del\), respectively.  Secondly, whilst the original papers of Hitchin and Donaldson considered only the case where \(\M\) is a manifold, their definitions are equally valid in the case where \(\M\) is an orbifold and accordingly one may consider the functionals \(\CH^\del\), \(\CCH^\del\), \(\SCH^\del\) and \(\SCCH^\del\) in the orbifold setting.  I remark that, whilst the functional \(\CCH^\del = \CCH\) has been considered in the case of closed manifolds \cite{TGo3Fi6&7D, SF&SM, STBoaMLCoG2S}, it has not, to the author's knowledge, been previously considered in the case of manifolds with boundary or in the case of orbifolds and, moreover, neither of the functionals \(\SCH^\del\) and \(\SCCH^\del\) have been previously considered in the literature.

The main result of this paper is the following.

\begin{Thm}\label{UB-Thm}
The functionals \(\CCH^\del\), \(\SCH^\del\) and \(\SCCH^\del\) are always unbounded above and below.  Specifically, let \(\M\) be a compact, oriented 7-manifold (or, more generally, 7-orbifold) with boundary and suppose that \(\M\) admits a closed \g\ 4-form \(\ps\).  Then, defining \([\ps]^\del_+\) by analogy with \([\ph]^\del_+\) above:
\ew
\inf_{\ps' \in [\ps]_+^\del} \CCH^\del\lt(\ps'\rt) = 0 \et \sup_{\ps' \in [\ps]^\del_+} \CCH^\del(\ps') = \0.
\eew
Likewise, the analogous statement holds for the functionals \(\SCH^\del\) and \(\SCCH^\del\).  In particular, by taking \(\del\M = \es\), the same conclusions apply to the Hitchin functionals \(\CCH\), \(\SCH\) and \(\SCCH\).
\end{Thm}

By contrast, the question of whether \(\CH^\del\) (or even \(\CH\)) is unbounded above remains open (although recent work by the author \cite{URftHFoG23F} has proven the unboundedness above of \(\CH\) in some specific examples).  Nevertheless, a partial analogue of \tref{UB-Thm} can obtained for the functional \(\CH^\del\):
\begin{Thm}\label{MT'}
The functional \(\CH^\del\) is always unbounded below.  Specifically, let \(\M\) be a compact, oriented 7-manifold (or, more generally, 7-orbifold) with boundary, let \(\ph\) be a closed \g\ 3-form on \(\M\).  Then:
\ew
\inf_{\ph' \in [\ph]^\del_+} \CH^\del\lt(\ph'\rt) = 0.
\eew
In particular, by taking \(\del\M = \es\), the same conclusion applies to the Hitchin functional \(\CH\).
\end{Thm}

The proof of \tref{UB-Thm} has two notable consequences.  Firstly, it enables the critical points of the functionals \(\CCH^\del\), \(\SCH^\del\) and \(\SCCH^\del\) to be classified.

\begin{Thm}\label{saddle-thm}
The critical points of \(\CCH^\del\), \(\SCH^\del\) and \(\SCCH^\del\) are always saddles.  Specifically, let \(\M\) be a compact, oriented 7-manifold (or, more generally, 7-orbifold) with (possibly empty) boundary and let \(\ps\) be a torsion-free \g\ 4-form on \(\M\).  Then there exist infinite-dimensional subspaces \(\mc{S}_4^\pm\lt(\ps\rt) \pc \T_\ps [\ps]_+^\del\) along which \(\lt.\mc{D}^2\CCH^\del\rt|_\ps\) is positive definite and negative definite, respectively.  The analogous statement holds for the functionals \(\SCH^\del\) and \(\SCCH^\del\).
\end{Thm}

I remark that, in contrast, the corresponding result for \(\CH^\del\) is known to be false in general.  E.g. in the case where \(\M\) is a closed manifold, the critical points are known to be local maxima by \cite{TGo3Fi6&7D} (see \cite[Prop.\ 1]{BVPiD74&3RtEH} and \cite{AEBVPfG2S} for a discussion of the case where \(\M\) has non-empty boundary).  It is this difference between \(\CH^\del\) and the functionals \(\CCH^\del\), \(\SCH^\del\) and \(\SCCH^\del\) which means that, currently, the unboundedness above of the functional \(\CH^\del\) remains an open problem.

Secondly, let \(\M\) be an oriented 7-manifold.  Recall that, given a closed \g\ 4-form \(\ps\) on \(\M\), the Laplacian coflow of \(\ps\) is defined to be the solution of the evolution PDE:
\ew
\frac{\del \ps(t)}{\del t} = \De_{\ps(t)}\ps(t) = \dd\dd^*_{\ps(t)}\ps(t) \et \ps(0) = \ps.
\eew
(Note that I adopt the sign convention for Laplacian coflow used in \cite{STBoaMLCoG2S}, rather than that used in the the original paper \cite{SSftLCosG2SwS}.)  Whilst the existence and uniqueness of the Laplacian coflow have yet to be proven, the flow can be regarded as the gradient flow of the Hitchin functional \(\CCH\) \cite{STBoaMLCoG2S}.  Consequently, \trefs{UB-Thm} and \ref{saddle-thm} intuitively suggest that most solutions of the Laplacian coflow on a given manifold \(\M\) (when they exist) will not converge to a torsion-free \g\ 4-form as \(t \to \infty\).  The following result confirms this expectation:
\begin{Thm}\label{donco-thm}
Let \(\M\) be a (possibly non-compact) oriented 7-manifold without boundary and let \(\ps \in \Om^4_+(\M)\) be a closed \g\ 4-form.  Consider the space:\vs{-2mm}
\ew
\mc{O}_{[\ps]_+} = \bigg\{\ps'\in[\ps]_+~\bigg| \parbox{7.5cm}{\begin{center}no solution to the Laplacian coflow started at \(\ps'\) converges to a torsion-free \g\ 4-form\end{center}}\bigg\}.
\eew
Then \(\mc{O}_{[\ps]_+} \pc [\ps]_+\) is dense in the \(C^0\) topology.
\end{Thm}

The results of this paper were obtained during the author's doctoral studies, which where supported by EPSRC Studentship 2261110.

\section{Preliminaries}

\subsection{Orbifolds}

The material in this subsection is largely based on \cite[\S\S1.1--1.3]{O&ST}, \cite[\S14.1]{THKLFPFftScDO} and \cite[\S1]{OOC}.  Let \(\M\) be a topological space and write:
\ew
\bb{R}^n_0 = \lt\{ \lt(x^1,...,x^n\rt) \in \bb{R}^n ~\m|~ x^n = 0 \rt\}, \hs{3mm} \bb{R}^n_+ = \lt\{ \lt(x^1,...,x^n\rt) \in \bb{R}^n ~\m|~ x^n \ge 0 \rt\}.
\eew
\begin{Defn}
An \(n\)-dimensional orbifold chart with boundary is the data \(\Xi = \lt(U,\Ga,\tld{U},\ch\rt)\) of a connected, open neighbourhood \(U\) in \(\M\), a finite subgroup \(\Ga \pc \GL(n;\bb{R})\), a connected, \(\Ga\)-invariant open neighbourhood \(\tld{U}\) of \(0\) in either \(\bb{R}^n\) or \(\bb{R}^n_+\) and a homeomorphism \(\ch: \lqt{\tld{U}}{\Ga} \to U\).  Write \(\tld{\ch}\) for the composite \(\tld{U} \oto{quot} \lqt{\tld{U}}{\Ga} \oto{\ch} U\).  Say that \(\Xi\) is centred at \(x \in \M\) if \(x = \tld{\ch}(0)\).   In this case, \(\Ga\) is called the orbifold group of \(x\), denoted \(\Ga_x\).  \(x\) is called a smooth point if \(\Ga_x = 0\), and a singular point if \(\Ga_x \ne 0\).

Now consider two orbifold charts with boundary \(\Xi_1 = \lt(U_1,\Ga_1,\tld{U}_1,\ch_1\rt)\) and \(\Xi_2 = \lt(U_2,\Ga_2,\tld{U}_2,\ch_2\rt)\) of dimensions \(n\) and \(m\), with \(U_1 \cc U_2\).   An embedding of \(\Xi_1\) into \(\Xi_2\) is the data of a smooth, open embedding \(\io_{12}:\tld{U}_1 \emb \tld{U}_2\) and a group isomorphism \(\la_{12}: \Ga_1 \to \Stab_{\Ga_2}(\io_{12}(0))\) such that \(\io_{12}(\si \cdot x) = \la_{12}(\si) \cdot \io_{12}(x)\) for all \(x \in \tld{U}_1\) and all \(\si \in \Ga_1\), and such that the following diagram commutes:
\ew
\bcd[column sep = 13mm, row sep = 3mm]
\tld{U}_1 \ar[rr, ^^22 \mla{\io_{12}} ^^22] \ar[dd, ^^22 \mla{\tld{\ch}_1} ^^22] & & \tld{U}_2 \ar[dd, ^^22 \mla{\tld{\ch}_2} ^^22]\\
& & \\
U_1 \ar[rr, ^^22\mla{incl}^^22, hook] &  & U_2
\ecd
\eew
Additionally, if \(\tld{U}_1 \cc \bb{R}^n_+\), require that \(\tld{U}_2 \cc \bb{R}^m_+\) and:
\ew
\io_{12}\lt(\tld{U}_1 \cap \bb{R}^n_0\rt) \cc \tld{U}_2 \cap \bb{R}^m_0.
\eew
In particular, note that whenever \(\Xi_1 \emb \Xi_2\), the dimensions \(n\) and \(m\) must necessarily be equal.

Now let \(\Xi_1\) and \(\Xi_2\) be arbitrary.  \(\Xi_1\) and \(\Xi_2\) are compatible if for every \(x \in U_1 \cap U_2\), there exists a chart \(\Xi_x = \lt(U_x,\Ga_x,\tld{U}_x,\ch_x\rt)\) centred at \(x\) together with embeddings \((\io_{x1},\la_{x1}): \Xi_x \emb \Xi_1\) and \((\io_{x2},\la_{x2}): \Xi_x \emb \Xi_2\).  If \(U_1 \cap U_2 = \es\), then \(\Xi_1\) and \(\Xi_2\) are automatically compatible, however if \(U_1 \cap U_2 \ne \es\) and \(\Xi_1\) and \(\Xi_2\) are compatible, then \(\Xi_1\) and \(\Xi_2\) have the same dimension; moreover, if \(\Xi_1\) and \(\Xi_2\) are centred at the same point \(x \in \M\), then \(\Ga_1 \cong \Ga_2\); thus, the orbifold group \(\Ga_x\) is well-defined up to isomorphism.

An atlas for \(\M\) is a collection \(\fr{A}\) of compatible orbifold charts with boundary which is maximal in the sense that if a chart \(\Xi\) is compatible with every chart in \(\fr{A}\), then \(\Xi \in \fr{A}\).  An orbifold with boundary is a connected, Hausdorff, second-countable topological space \(\M\) equipped with an atlas \(\fr{A}\).  Every chart of \(\M\) has the same dimension \(n\); term this common dimension the dimension of the orbifold.
\end{Defn}

If \(\Ga_x = {\bf 1}\) for all \(x \in \M\), then \(\M\) may naturally be regarded as a manifold; in this way, all of the results obtained in this paper for orbifolds automatically imply the corresponding result for manifolds.  More generally, given an arbitrary orbifold \(\M\), the set of all points \(x \in \M\) such that \(\Ga_x = {\bf 1}\) forms an open and dense subset of \(\M\), denoted \(\M_{smooth}\).  Its connected components naturally inherit the structure of smooth manifolds with boundary.  Its complement is known as the singular locus of \(\M\), denoted \(\M_{sing}\), and has positive codimension in \(\M\).  Likewise, the set of all points \(x \in \M\), such that there exists an orbifold chart with boundary \(\Xi = \lt(U,\Ga,\tld{U},\ch\rt)\) centred at \(x\) with \(\tld{U} \cc \bb{R}^n_+\), forms a closed subset termed the boundary of \(\M\), denoted \(\del\M\).  Its connected components naturally inherit the structure of orbifolds without boundary, i.e.\ orbifolds whose boundary is empty.

Let \(\M\) be an \(n\)-orbifold with boundary.  Given any chart \(\Xi = \lt(U, \Ga, \tld{U}, \ch\rt)\) for \(\M\), the action of \(\Ga\) on \(\tld{U}\) naturally lifts to an action of \(\Ga\) on \(\T\tld{U}\) by bundle automorphisms.  Given a second chart \(\Xi' = (U',\Ga',\tld{U}',\ch')\) embedding into \(\Xi\), the map \(\tld{U}' \emb \tld{U}\) induces an equivariant embedding of bundles \(\T\tld{U}' \emb \T\tld{U}\).  Define:
\ew
\T\M = \rqt{\lt[\coprod_{\Xi} \lt(\lqt{\T\tld{U}}{\Ga}\rt)\rt]}{\tl}
\eew
where the quotient by \(\tl\) denotes that one should glue along all embeddings \(\T\tld{U}' \emb \T\tld{U}\).  The resulting space \(\T\M\) is an orbifold with boundary, and the natural projection \(\pi: \T\M \to \M\) gives \(\T\M\) the structure of an orbifold vector bundle over \(\M\) (see \cite[p.\ 173]{THKLFPFftScDO} and \cite[p.\ 302]{OOC}).  \(\T\M\) is termed the tangent bundle of \(\M\).  Given \(x \in \M\), the tangent space at \(x\) denoted \(\T_x\M\) is the preimage of \(x\) under the map \(\T \M \to \M\).  It may be identified with the quotient space \(\lqt{\bb{R}^n}{\Ga_x}\), where \(\Ga_x\) is the orbifold group at \(x\).  A section of the tangent bundle is then simply a continuous map \(X: \M \to \T \M\) such that, for each local chart \(\Xi = \lt(U, \Ga, \tld{U}, \ch\rt)\) for \(\M\), there is a smooth, \(\Ga\)-invariant, vector field \(X': \tld{U} \to \T \tld{U}\) such that the following diagram commutes:
\e\label{sec-lift}
\bcd[column sep = 13mm, row sep = 3mm]
\tld{U} \ar[rr, ^^22 \mla{X'} ^^22] \ar[dd, ^^22 \mla{\tld{\ch}} ^^22] & & \T \tld{U} \ar[dd]\\
& &\\
U \ar[rr, ^^22 \mla{X} ^^22] &  & \pi^{-1}(U)
\ecd
\ee
In particular, \(X'|_0\) is invariant under the action of \(\Ga\), i.e.\ \(\Ga \cc \Stab_{\GL_+(n;\bb{R})}(X'|_0)\).  In a similar way, one may define the cotangent bundle of an orbifold with boundary, tensor bundles, bundles of exterior forms etc., with sections of these bundles defined analogously.  \(\M\) is orientable if \(\ww{n}\T^* \M\) is trivial, i.e.\ has a non-vanishing section (in particular, all the orbifold groups of \(\M\) lie in \(\GL_+(n;\bb{R})\)); an orientation is simply a choice of trivialisation.  Riemannian metrics and other geometric structures can be defined similarly.

\subsection{\g- and \sg-structures on orbifolds}\label{(S)G2-prelim}

Let \(\si_0 \in \ww{p}\lt(\bb{R}^n\rt)^*\) and let \(\M\) be an oriented \(n\)-orbifold, possibly with boundary.  A \(p\)-form \(\si \in \Om^p(\M)\) is termed a \(\si_0\)-form if, for each chart \(\Xi = \lt(U,\Ga,\tld{U},\ch\rt)\) for \(\M\), there exists a \(\Ga\)-invariant lift \(\si'\) of \(\si\) to \(\Om^p\lt(\tld{U}\rt)\) (as in \eref{sec-lift}) with the property that, for each \(x \in \tld{U}\), there exists an orientation-preserving isomorphism \(\al: \T_x\tld{U} \to \bb{R}^n\) satisfying \(\al^*\si_0 = \si'|_x\).  \(\si_0\)-forms are termed stable if the \(\GL_+(n;\bb{R})\)-orbit of \(\si_0\) in \(\ww{p}\lt(\bb{R}^n\rt)^*\) is open; in this case, all sufficiently small perturbations of \(\si_0\)-forms are also \(\si_0\)-forms.

In this terminology, \g\ 3-forms are simply \(\vph_0\)-forms, for \(\vph_0\) as defined in the introduction.  As stated in the introduction, \g\ 3-forms are stable.  Indeed, since the stabiliser of \(\vph_0\) in \(\GL_+(7;\bb{R})\) is isomorphic to \g, it is 14-dimensional and thus the \(\GL_+(7;\bb{R})\)-orbit of \(\vph_0\) in \(\ww{3}\lt(\bb{R}^7\rt)^*\) has dimension:
\ew
\dim\GL_+(7;\bb{R}) - \dim\Gg_2 = 49 - 14 = 35 = \dim\ww{3}\lt(\bb{R}^7\rt)^*,
\eew
implying stability, as claimed.  Given a \g\ 3-form \(\ph\in\Om^3_+(\M)\), the 4-form \(\Hs_\ph\ph\) is automatically a \g\ 4-form, i.e.\ a \(\vps_0\)-form for:
\ew
\vps_0 = \Hs_0\vph_0 = \th^{4567} + \th^{2367} + \th^{2345} + \th^{1357} - \th^{1346} - \th^{1256} - \th^{1247} \in\ww{4}(\bb{R}^7)^*,
\eew
where \(\Hs_0\) denotes the Hodge star on the vector space \(\bb{R}^7\), defined by the Euclidean inner product and the usual choice of orientation.  The stabiliser of \(\vps_0\) in \(\GL_+(7;\bb{R})\) is also \g\ and thus \g\ 4-forms are also stable.  Moreover, \g-structures on oriented orbifolds can equivalently be characterised by either \g\ 3-forms or \g\ 4-forms (see \cite[p.\ 4]{OOC} for a definition of principal bundles on orbifolds).  Write \(\ww[+]{p}\T^*\M\) for the bundle of \g\ \(p\)-forms on \(\M\) (\(p = 3,4\)) and \(\Om^p_+\) for its corresponding sheaf of sections.

Likewise, consider the 3-form:
\ew
\svph_0 = \th^{123} - \th^{145} - \th^{167} + \th^{246} - \th^{257} - \th^{347} - \th^{356} \in\ww{3}(\bb{R}^7)^*,
\eew
where \(\svph_0\) differs from \(\vph_0\) only in the sign of the second and third terms.  \(\svph_0\)-forms are termed \sg\ 3-forms.  The stabiliser of \(\svph_0\) in \(\GL_+(7;\bb{R})\) is isomorphic to \sg\ and hence \sg\ 3-forms are equivalent to \sg-structures on \(\M\).  Since \sg\ is also 14-dimensional, \sg\ 3-forms are also stable.  Moreover, since \(\tld{\Gg}_2 \pc \SO(3,4)\) preserves the indefinite inner product \(\tld{g}_0 = \sum_{i=1}^3 \lt(\th^i\rt)^{\ts2} - \sum_{i=4}^7 \lt(\th^i\rt)^{\ts2}\) and Euclidean volume element \(vol_0\) on \(\bb{R}^7\), \(\sph\) induces a metric \(g_{\sph}\) of signature \((3,4)\) and a volume form \(vol_{\sph}\), and hence a Hodge star operator \(\Hs_{\sph}\) (see \cite[Ch.\ 1.C]{EM} for an exposition of the elementary properties of pseudo-Riemannian metrics).  The 4-form \(\Hs_{\sph}\sph\) is automatically a \(\svps_0\)-form for:
\ew
\svps_0 = \tld{\Hs}_0\svph_0 = \th^{4567} - \th^{2367} - \th^{2345} + \th^{1357} - \th^{1346} - \th^{1256} - \th^{1247} \in\ww{4}(\bb{R}^7)^*,
\eew
where \(\tld{\Hs}_0\) denotes the Hodge star on the vector space \(\bb{R}^7\), defined by the indefinite inner product \(\tld{g}_0\) and the usual choice of orientation.  Such 4-forms are termed \sg\ 4-forms and are, again, stable and equivalent to \sg-structures on \(\M\).  Write \(\ww[\tl]{p}\T^*\M\) for the bundle of \sg\ \(p\)-forms on \(\M\) (\(p = 3,4\)) and \(\Om^p_{\tl}\) for its corresponding sheaf of sections.

Next, consider the action of \g\ on \(\bb{R}^7\).  The spaces \(\ww{0}\lt(\bb{R}^7\rt)^*\), \(\ww{1}\lt(\bb{R}^7\rt)^*\), \(\ww{6}\lt(\bb{R}^7\rt)^*\) and \(\ww{7}\lt(\bb{R}^7\rt)^*\) are simple \g-modules, however the remaining exterior powers of \(\lt(\bb{R}^7\rt)^*\) are reducible, decomposing into simple modules as follows:
\e\label{G2TD}
\ww{2}\lt(\bb{R}^7\rt)^* &= \ww[7]{2}\lt(\bb{R}^7\rt)^* \ds \ww[14]{2}\lt(\bb{R}^7\rt)^*;\\
\ww{3}\lt(\bb{R}^7\rt)^* &= \ww[1]{3}\lt(\bb{R}^7\rt)^* \ds \ww[7]{3}\lt(\bb{R}^7\rt)^* \ds \ww[27]{3}\lt(\bb{R}^7\rt)^*;\\
\ww{4}\lt(\bb{R}^7\rt)^* &= \ww[1]{4}\lt(\bb{R}^7\rt)^* \ds \ww[7]{4}\lt(\bb{R}^7\rt)^* \ds \ww[27]{4}\lt(\bb{R}^7\rt)^*;\\
\ww{5}\lt(\bb{R}^7\rt)^* &= \ww[7]{5}\lt(\bb{R}^7\rt)^* \ds \ww[14]{5}\lt(\bb{R}^7\rt)^*,
\ee
where the subscript in each case denotes the dimension of the module and:
\e\label{G2TDEx}
\begin{aligned}
\ww[7]{2}\lt(\bb{R}^7\rt)^* &= \lt\{v \hk \vph_0 ~\m|~ v \in \bb{R}^7 \rt\} \cong \bb{R}^7;\\
\ww[14]{2}\lt(\bb{R}^7\rt)^* &= \lt\{\al \in \ww{2}\lt(\bb{R}^7\rt)^*~\m|~\al \w \vps_0 = 0\rt\} \cong \fr{g}_2;\\
\ww[1]{3}\lt(\bb{R}^7\rt)^* &= \bb{R} \1 \vph_0 \cong \bb{R};\\
\ww[7]{3}\lt(\bb{R}^7\rt)^* &= \lt\{v \hk \vps_0 ~\m|~ v\in \bb{R}^7 \rt\} \cong \bb{R}^7;\\
\ww[27]{3}\lt(\bb{R}^7\rt)^* &= \lt\{a \in \ww{3}\lt(\bb{R}^7\rt)^* ~\m|~ a\w \vph_0 = 0 \,\text{and}\, a \w \vps_0 = 0\rt\} \cong \ss[0]{2}\lt(\bb{R}^7\rt)^*.
\end{aligned}
\ee
(The notation \(\ss[0]{2}\lt(\bb{R}^7\rt)^*\) refers to the space of trace-free symmetric bilinear forms on \(\bb{R}^7\), where the trace is computed using the metric \(g_0\).)  Any two simple modules of a given dimension are isomorphic; in particular:
\ew
\Hs_0:\ww[q]{p}\lt(\bb{R}^7\rt)^* \to \ww[q]{7-p}\lt(\bb{R}^7\rt)^*
\eew
is a \g-equivariant isomorphism for each \(p = 0,...,7\) and corresponding \(q\); thus explicit expressions for all the simple modules occurring in \eref{G2TD} can be deduced from \eref{G2TDEx} via Hodge star.  As is customary, I write \(\pi_d:\ww{\pt}\lt(\bb{R}^7\rt)^*\to\ww[d]{\pt}\lt(\bb{R}^7\rt)^*\) for the \(g_0\)-orthogonal projection; since for each exterior power no subscript occurs more than once, no ambiguity should arise from this notation.  Exterior forms lying in a particular space \(\ww[q]{p}\lt(\bb{R}^7\rt)^*\) are said to be of definite type, and for an arbitrary \(p\)-form \(\si\) the decomposition \(\si = \sum_q \pi_q(\si)\) is called the type decomposition of \(\si\).  Since the representation theories of the groups \g\ and \sg\ are identical, each coinciding with the representation theory of \(\fr{g}_{2,\bb{C}}\), \sg\ also induces a type decomposition of exterior forms which is entirely analogous to the \g\ type decomposition described above.

Now let \(\M\) be an oriented 7-orbifold (possibly with boundary) and let \(\ph\) be a \g\ or \sg\ 3-form on \(\M\) (the case of 4-forms is similar).  Given a chart \(\Xi = \lt(U, \Ga, \tld{U}, \ch\rt)\) for \(\M\), writing \(\ph'\) for the lift of \(\ph\) to \(\tld{U}\) as in \eref{sec-lift}, \(\ph'\) induces a \(\Ga\)-invariant decomposition \(\ww{p}\T^*\tld{U} \cong \Ds_q \ww[q]{p}\T^*\tld{U}\) and hence one naturally obtains a decomposition \(\ww{p}\T^*\M = \Ds_q \ww[q]{p}\T^*\M\), where each \(\ww[q]{p}\T^*\M\) is an orbifold vector bundle over \(\M\), with fibre over \(x\) isomorphic to \(\lqt{\ww[q]{p}\lt(\bb{R}^7\rt)^*}{\Ga_x}\) (which is well-defined, since \(\Ga_x\) is isomorphic to a subgroup of \g\ or \sg\ as appropriate; see the discussion immediately after \eref{sec-lift}).  Thus, \(\ph\) naturally induces a type decomposition on the exterior forms of \(\M\).

Finally, recall the functional \(\CH^\del\) defined in the introduction.  In a similar way, given a closed \g\ 4-form \(\ps\), define:
\e\label{be+-defn}
[\ps]_+^\del = \lt\{ \ps' \in \ps + \dd\Om^3(\M;\del\M) ~\m|~ \ps' \text{ is a \g\ 4-form}\rt\}.
\ee
Likewise, given a closed \sg\ 3-form \(\sph\) or \sg\ 4-form \(\sps\), define \sg-analogues \(\lt[\sph\rt]_{\tl}^\del\) and \(\lt[\sps\rt]_{\tl}^\del\) of \([\ph]_+^\del\) and \([\ps]_+^\del\), respectively.  One can then define functionals \(\CCH^\del\), \(\SCH^\del\) and \(\SCCH^\del\) on \([\ps]_+^\del\), \(\lt[\sph\rt]_{\tl}^\del\) and \(\lt[\sps\rt]_{\tl}^\del\) by integrating \(vol_\si\) over \(\M\), for \(\si \in [\ps]_+^\del\), \(\lt[\sph\rt]_{\tl}^\del\) or \(\lt[\sps\rt]_{\tl}^\del\) as appropriate.  Since \g\ 3-forms are stable and \(\M\) is compact, the subset \([\ph]_+^\del \pc \ph + \dd\Om^2(\M;\del\M)\) is open in the \(C^0\)-topology and thus one can identify \(\T_{\ph'}[\ph]_+^\del \cong \dd\Om^2(\M;\del\M)\).  A similar argument applies to \([\ps]_+^\del\), \(\lt[\sph\rt]_{\tl}^\del\) and \(\lt[\sps\rt]_{\tl}^\del\).  Using type decomposition, one can explicitly compute the functional derivatives of \(\CH^\del\), \(\CCH^\del\), \(\SCH^\del\) and \(\SCCH^\del\).  This was accomplished for \(\CH^\del = \CH\) in the case of closed manifolds by Hitchin in \cite[Thm.\ 19 \& Lem.\ 20]{TGo3Fi6&7D}\footnote{Note that the formulae for \(\lt.\mc{D}\CH^\del\rt|_\ph\) and \(\lt.\mc{D}^2\CH^\del\rt|_\ph\) above differ from the corresponding formulae in \cite{TGo3Fi6&7D}, as the author of this paper has discovered an error in the numerical factor of \(\frac{7}{18}\) used {\it op.\ cit.}, which has been corrected to \(\frac{1}{3}\) in the formulae here presented.} via arguments which apply equally to compact orbifolds with boundary.  Analogous arguments can be used to compute the first and second derivatives of \(\CCH^\del\); see \cite[Thm.\ 1]{SF&SM} and also \cite[Prop.\ 3.3 \& 3.4]{STBoaMLCoG2S} for the case of closed manifolds.  Since the computation of these derivatives is completely representation-theoretic, and since \g\ and \sg\ have identical representation theories, the formulae for \(\SCH^\del\) and \(\SCCH^\del\) are identical.  Thus one obtains:
\begin{Prop}\label{Hitchin-derivs}
The first and second derivatives of \(\CH^\del\) (equivalently \(\SCH^\del\)) are given by:
\ew
\bcd[row sep = -2mm, column sep = 0pt]
\mc{D}\overset{(\tl)}{\mathcal{H}_3^\del}|_\ph:\dd\Om^2(\M;\del\M) \ar[r] & \bb{R} & & \mc{D}^2\overset{(\tl)}{\mathcal{H}_3^\del}|_\ph:\dd\Om^2(\M;\del\M) \x \dd\Om^2(\M;\del\M) \ar[r] & \bb{R}\\
\si \ar[r, maps to] & \frac{1}{3}\bigint_\M \si \w \Hs_\ph \ph & & (\si_1,\si_2) \ar[r, maps to] & \frac{1}{3}\bigint_\M \si_1 \w \Hs_\ph I_\ph(\si_2)
\ecd
\eew
and the first and second derivatives of \(\CCH^\del\) (equivalently \(\SCCH^\del\)) are given by:
\ew
\bcd[row sep = -2mm, column sep = 0pt]
\mc{D}\overset{(\tl)}{\mathcal{H}_4^\del}|_\ps:\dd\Om^3(\M;\del\M) \ar[r] & \bb{R} & & \mc{D}^2\overset{(\tl)}{\mathcal{H}_4^\del}|_\ps:\dd\Om^3(\M;\del\M) \x \dd\Om^3(\M;\del\M) \ar[r] & \bb{R}\\
\vpi \ar[r, maps to] & \frac{1}{4}\bigint_\M \vpi \w \Hs_\ps\ps & & (\vpi_1,\vpi_2) \ar[r, maps to] & \frac{1}{4}\bigint_\M \vpi_1 \w \Hs_\ps J_\ps(\vpi_2)
\ecd
\eew
where:
\ew
I_\ph(\si) = \frac{4}{3}\pi_1(\si) + \pi_7(\si) - \pi_{27}(\si) \et J_\ps(\si) = \frac{3}{4}\pi_1(\si) + \pi_7(\si) - \pi_{27}(\si).
\eew
(Here, the projections \(\pi_\pt\) are induced by \(\ph\) and \(\ps\), respectively.)
\end{Prop}
In particular, Stokes' Theorem shows that, as stated in the introduction, \(\ph'\) is a critical point of \(\CH^\del\) \iff\ it satisfies \(\dd\Hs_{\ph'}\ph' = 0\), i.e.\ it is torsion-free.  Likewise, the critical points of the functionals \(\CCH^\del\), \(\SCH^\del\) and \(\SCCH^\del\) are torsion-free, corresponding -- in the manifold case -- to metrics with holonomy contained in \g\ or \sg, as appropriate.

\section{Local volume-altering perturbations of closed \g\ 4-forms}\label{4Loc}

View \(\bb{R}^7\) as a manifold with canonical coordinates \((x^1,...,x^7)\) and let \(\B_\eta\) denote the open ball of Euclidean radius \(\eta>0\) about \(0\).  Initially consider the (torsion-free) \g\ 4-form on \(\bb{R}^7\) defined by:
\ew
\ps_0 = \dd x^{4567} + \dd x^{2367} + \dd x^{2345} + \dd x^{1357} - \dd x^{1346} - \dd x^{1256} - \dd x^{1247},
\eew
inducing the Euclidean metric \(g_0\) and volume form \(vol_0 = \dd x^{1...7}\).  Restricting \(\ps_0\) to \(\ol{\B}_\eta\), write \(\CCH[,\ol{\B_\eta}]^\del\) for the corresponding Donaldson--Hitchin functional on \([\ps_0]_+^\del\).  For brevity of notation, I use the symbol \(\gl\) by analogy with \(\pm\) to mean `greater than or less than, respectively' where, by convention, the upper-most symbol should be read first (thus, the equation \(a \pm b \ \gl \ c\) should be interpreted as the pair of equations \(a + b > c\) and \(a - b < c\)); define \(\gle\) analogously.  The key result upon which this paper's investigation of \(\CCH^\del\) is founded is the following:
\begin{Lem}\label{P0}~

\begin{enumerate}
\item There exist \(\al^\pm \in \Om^3\lt(\ol{\B_1}\rt)\) with \(\supp(\al^\pm) \C \B_1\) such that
\ew
\lt.\mc{D}^2\CCH[,\ol{\B}_1]^\del\rt|_{\ps_0}(\dd\al^\pm,\dd\al^\pm) \gl 0.
\eew

\item
There exist \(\be^\pm\in\Om^3\lt(\ol{\B}_1\rt)\) with \(\supp(\be^\pm)\C\B_1\) such that:
\e\label{c1}
\ps_0 + \dd\be^\pm \text{ is a } G_2 \text{ 4-form}
\ee
and
\ew
\CCH[,\ol{\B}_1]^\del\lt(\ps_0 + \dd\be^\pm\rt) \gl \CCH[,\ol{\B}_1]^\del(\ps_0).
\eew
\end{enumerate}
\end{Lem}

\begin{proof}
Using \pref{Hitchin-derivs}, the Hessian of \(\CCH[,\ol{\B}_1]^\del\) is given by:
\e\label{D2CCH-eq}
\lt.\mc{D}^2\CCH[,\ol{\B}_1]^\del\rt|_{\ps_0}(\ga_1,\ga_2) = \bigintsss_{\ol{\B}_1} \frac{1}{4}\lt(\frac{3}{4}g_0\lt(\pi_1\ga_1,\pi_1\ga_2\rt) + g_0\lt(\pi_7\ga_1,\pi_7\ga_2\rt) - g_0\lt(\pi_{27}\ga_1,\pi_{27}\ga_2\rt)\rt) vol_0,
\ee
where \(\pi_\pt\) denotes the type decomposition \wrt\ \(\ps_0\).  For (1), consider the choices:
\ew
\al^+ = f(r)\cdot \Hs_0\ps_0 \et \al^- = f(r)\cdot \dd x^{123},
\eew
where \(r = \sqrt{(x^1)^2 + ... + (x^7)^2}\) and \(f:[0,1]\to[0,1]\) is a smooth function such that \(f\equiv 1\) on a neighbourhood of \(0\) and \(f \equiv 0\) on a neighbourhood of \(1\). Then, since \(\dd\Hs_0\ps_0 = 0\), one finds \(\dd\al^+ = \dd f \w \Hs_0\ps_0 \in \Om^4_7\lt(\ol{\B}_1\rt)\) (see \eref{G2TDEx}).  Hence using \eref{D2CCH-eq}:
\ew
\lt.\mc{D}^2\CCH[,\ol{\B}_1]^\del\rt.|_{\ps_0}(\dd\al^+,\dd\al^+) = \bigintsss_{\ol{\B}_1} \frac{1}{4}\lt\|\dd\al^+\rt\|^2_{g_0} vol_0 > 0,
\eew
as required.  For \(\al^-\), a direct calculation yields:
\caw
\|\pi_1(\dd\al^-)\|^2_{g_0} = 0, \hs{5mm} \|\pi_7(\dd\al^-)\|^2_{g_0} = \frac{1}{4}\lt(\frac{\dd f}{\dd r}\frac{1}{r}\rt)^2\lt(\lt(x^4\rt)^2 + \lt(x^5\rt)^2 + \lt(x^6\rt)^2 + \lt(x^7\rt)^2\rt),\\
\|\pi_{27}(\dd\al^-)\|^2_{g_0} = \frac{3}{4}\lt(\frac{\dd f}{\dd r}\frac{1}{r}\rt)^2\lt(\lt(x^4\rt)^2 + \lt(x^5\rt)^2 + \lt(x^6\rt)^2 + \lt(x^7\rt)^2\rt).
\caaw
Thus, for this choice of \(\al^-\), one computes using \eref{D2CCH-eq} that:
\ew
\lt.\mc{D}^2\CCH[,\ol{\B}_1]^\del\rt|_{\ps_0}(\dd\al^-,\dd\al^-) = -\bigintsss_{\ol{\B}_1}\frac{1}{8}\lt(\frac{\dd f}{\dd r}\frac{1}{r}\rt)^2\lt(\lt(x^4\rt)^2 + \lt(x^5\rt)^2 + \lt(x^6\rt)^2 + \lt(x^7\rt)^2\rt) vol_0 < 0.
\eew

For (2), take \(\be^\pm = t\al^\pm\) for some \(t > 0\) to be determined later.  Then \(\ps^\pm(t)\) satisfies \eref{c1} for all \(t\) sufficiently small, since \g\ 4-forms are stable.  Using \pref{Hitchin-derivs}, together with the fact that \(\ps_0\) (being torsion-free) is a critical point of the functional \(\CCH[,\ol{\B}_1]^\del\), one may Taylor expand:
\ew
\CCH[,\ol{\B}_1]^\del\lt(\ps_0 + \dd\be^\pm\rt) = \CCH[,\ol{\B}_1]^\del(\ps_0) + \frac{t^2}{2}\lt.\mc{D}^2\CCH[,\ol{\B}_1]^\del\rt|_{\ps_0}(\dd\al^\pm,\dd\al^\pm) + \mc{O}(t^3).
\eew
The result now follows by taking \(t > 0\) sufficiently small.

\end{proof}

Next, let \(\ps\) be a closed \g\ 4-form on \(\bb{R}^7\) and let \(\B_\eta(\ps)\) denote the open geodesic ball of radius \(\eta\) centred at 0, defined by the metric \(g_\ps\).
\begin{Lem}\label{P}
There exist \(\eta_0 > 0\) (depending on \(\ps\)) and \(\ep > 0\) (independent of \(\ps\)) such that for all \(\eta \in (0,\eta_0]\):

\begin{enumerate}
\item There exist \(\al^\pm_\eta \in \Om^3\lt(\ol{\B}_\eta(\ps)\rt)\) with \(\supp(\al^\pm_\eta) \C \B_\eta(\ps)\) such that
\ew
\lt.\mc{D}^2\CCH[,\ol{\B}_\eta(\ps)]^\del\rt|_\ps(\dd\al^\pm_\eta,\dd\al^\pm_\eta) \gl 0.
\eew

\item
There exist \(\be^\pm_\eta\in\Om^3\lt(\ol{\B}_\eta(\ps)\rt)\) with \(\supp(\be^\pm_\eta)\C\B_\eta(\ps)\) such that \(\ps + \dd\be^\pm_\eta\) is a \g\ 4-form and:
\ew
\CCH[,\ol{\B}_\eta(\ps)]^\del\lt(\ps + \dd\be^\pm_\eta\rt) \gle (1\pm\ep)\CCH[,\ol{\B}_\eta(\ps)]^\del(\ps).
\eew
\end{enumerate}
\end{Lem}

\begin{proof}
Let \(\al^\pm\) and \(\be^\pm\) be as in \lref{P0} and extend them via 0 to forms on \(\bb{R}^7\).  Choose \(r_0 > 0\), \(\ep > 0\) sufficiently small such that \(\supp\lt(\al^\pm\rt)\) and \(\supp\lt(\be^\pm\rt) \C \B_{1-r_0}\), and:
\ew
\CCH[,\ol{\B}_{1 \mp r_0}]^\del\lt(\ps_0 + \be^\pm\rt) \gle (1 \pm 2\ep) \CCH[,\ol{\B}_{1 \pm r_0}]^\del(\ps_0).
\eew
Observe that the statements in (1) and (2) are diffeomorphism invariant.  Thus, by applying a suitable (linear) diffeomorphism of \(\bb{R}^7\), one may assume \wlg\ that \(\ps|_0 = \ps_0|_0\).  Moreover, note that (1) and (2) are invariant under the rescaling:
\ew
\eta \mt \la^\frac{1}{4}\eta, \hs{3mm} \ps \mt \la\ps, \hs{3mm} \al^\pm_\eta \mt \la\al^\pm_\eta  \et \be^\pm_\eta \mt \la\be^\pm_\eta
\eew
for any \(\la > 0\) (in particular, note that \(\B_{\la^\frac{1}{4}\eta}(\la\ps) = \B_\eta(\ps)\)).  Thus, for each \(\eta>0\) consider the diffeomorphism \(\mu_\eta : x \in \bb{R}^7 \mt \eta x \in \bb{R}^7\) and write \(\ps_\eta = \eta^{-4}\mu_\eta^*\ps\).  Then \(\ps_\eta \to \ps_0\) locally uniformly on \(\bb{R}^7\) as \(\eta \to 0\); in particular, there exists \(\eta_1 = \eta_1(\ps) > 0\) such that for all \(\eta < \eta_1\):
\ew
\lt.\mc{D}^2\CCH[,\ol{\B}_{1-r_0}]^\del\rt|_{\ps_\eta}(\dd\al^\pm,\dd\al^\pm) = \lt.\mc{D}^2\CCH[,\ol{\B}_1]^\del\rt|_{\ps_\eta}(\dd\al^\pm,\dd\al^\pm)\gl 0
\eew
(the first equality following from \(\supp\lt(\al^\pm\rt) \C \B_{1-r_0}\)):
\ew
\ps_\eta + \dd\be^\pm \text{ is a } \Gg_2 \text{ 4-form}
\eew
and:
\e\label{r0-eps}
\CCH[,\ol{\B}_{1 \mp r_0}]^\del\lt(\ps_\eta + \dd\be^\pm\rt) \gle (1 \pm \ep) \CCH[,\ol{\B}_{1 \pm r_0}]^\del(\ps_\eta).
\ee
Next, observe that \(\mu_\eta^{-1}\lt(\ol{\B}_\eta(\ps)\rt) = \eta^{-1}\ol{\B}_\eta(\ps)\) converges in the Hausdorff sense to the closed unit ball \(\ol{\B}_1\) as \(\eta \to 0\); explicitly, for each \(r \in (0,1)\), there exists \(\eta_2 = \eta_2(r,\ps) > 0\) such that for all \(\eta < \eta_2\):
\e\label{Hd-cvgce}
\ol{\B}_{1-r} \cc \eta^{-1}\ol{\B}_\eta(\ps) \cc \ol{\B}_{1+r}.
\ee
Set \(\eta_0 = \min[\eta_1(\ps), \eta_2(r_0,\ps)]\).  Then clearly \(\ps_\mu + \dd\be^\pm\) is a \g\ 4-form for all \(\eta < \eta_0\).   Moreover:
\ew
\lt.\mc{D}^2\CCH[,\eta^{-1}\ol{\B}_\eta(\ps)]^\del\rt|_{\ps_\eta}(\dd\al^\pm,\dd\al^\pm) = \lt.\mc{D}^2\CCH[\ol{\B}_{1-r_0}]^\del\rt|_{\ps_\eta}(\dd\al^\pm,\dd\al^\pm) \gl 0,
\eew
the first equality following from \(\supp\lt(\dd\al^\pm\rt) \C \ol{\B}_{1-r_0}\).  Finally:
\ew
\CCH[,\eta^{-1}\ol{\B}_\eta(\ps)]^\del\lt(\ps_\eta - \dd\be^-\rt) &\le \CCH[,\ol{\B}_{1+r_0}]^\del\lt(\ps_\eta - \dd\be^-\rt) \hs{3mm} \text{(by \eref{Hd-cvgce})}\\
&\le (1-\ep)\CCH[,\ol{\B}_{1 - r_0}]^\del(\ps_\eta) \hs{3mm} \text{(by \eref{r0-eps})}\\
&\le (1-\ep)\CCH[,\eta^{-1}\ol{\B}_\eta(\ps)]^\del(\ps_\eta) \hs{3mm} \text{(by \eref{Hd-cvgce})}
\eew
and likewise:
\ew
\CCH[,\eta^{-1}\ol{\B}_\eta(\ps)]^\del\lt(\ps_\eta + \dd\be^+\rt) &\ge \CCH[,\ol{\B}_{1-r_0}]^\del\lt(\ps_\eta + \dd\be^+\rt) \hs{3mm} \text{(by \eref{Hd-cvgce})}\\
&\ge (1+\ep)\CCH[,\ol{\B}_{1 + r_0}]^\del(\ps_\eta) \hs{3mm} \text{(by \eref{r0-eps})}\\
&\le (1+\ep)\CCH[,\eta^{-1}\ol{\B}_\eta(\ps)]^\del(\ps_\eta) \hs{3mm} \text{(by \eref{Hd-cvgce})}.
\eew
Thus, by diffeomorphism and scale invariance, the proof is completed by setting \(\al^\pm_\eta = \eta^4\mu_{\eta^{-1}}^*\al^\pm\) and \(\be^\pm_\eta = \eta^4\mu_{\eta^{-1}}^*\be^\pm\).

\end{proof}

\section{The unboundedness above and below of \(\CCH^\del\)}\label{GMT}

The aim of this section is to prove \tref{UB-Thm} in the \g\ case.  Recall the statement of the theorem:\vs{3mm}

\noindent{\bf Theorem \ref{UB-Thm}} (\g\ case){\bf .}
\em The functional \(\CCH^\del\) is always unbounded above and below.  Specifically, let \(\M\) be a compact, oriented 7-orbifold with boundary and let \(\ps\) be a closed \g\ 4-form on \(\M\).  Then:
\ew
\inf_{\ps' \in [\ps]_+^\del} \CCH^\del\lt(\ps'\rt) = 0 \et \sup_{\ps' \in [\ps]^\del_+} \CCH^\del(\ps') = \infty.
\eew
In particular, by taking \(\del\M = \es\), the same conclusion applies to the Hitchin functional \(\CCH\).\vs{2mm}\em

Initially, let \((\M,\ps)\) be a (possibly non-compact) oriented 7-manifold without boundary equipped with a closed \g\ 4-form and for each \(p \in \M\), let \(\B_r(p)\) denote the (open) geodesic ball of radius \(r\) centred at \(p\) defined by the metric \(g_\ps\).  By applying \lref{P} about each point \(p \in \M\), one immediately obtains the following result:
\begin{Lem}\label{local-ud}
Let \(\ep > 0\) be as in \lref{P}.  Then for each \(p\in \M\), there exists \(\eta_p > 0\) (depending on \(\ps\)) such that for all \(\eta\in(0,\eta_p]\), there exist \(\be_\eta^{p/\pm}\in\Om^3\lt(\M\rt)\) with \(\supp(\be_\eta^{p/\pm}) \C \B_\eta(p)\) satisfying:
\ew
\ps + \dd\be_\eta^{p/\pm} \text{ is a } G_2 \text{ 4-form}
\eew
and:
\ew
\CCH[,\ol{\B}_\eta(p)]^\del\lt(\ps + \dd\be_\eta^{p/\pm}\rt) \gle (1\pm\ep) \CCH[,\ol{\B}_\eta(p)]^\del(\ps).
\eew
\end{Lem}

To convert this local result into the global statement of the unboundedness of \(\CCH^\del\), I establish the following covering lemma:

\begin{Lem}\label{cov-lem}
Let \((\M,\ps)\) be a (possibly non-compact) oriented 7-manifold without boundary equipped with a closed \g\ 4-form, let \(\eta_p\) be as defined in \lref{local-ud} and consider the set:
\ew
\mc{G} = \lt\{ \ol{\B}_\eta(p) ~\m|~ p \in \M,\ \eta \in (0,\eta_p] \rt\}.
\eew
Then there exists a countable, disjoint subset \(\mc{F} \pc \mc{G}\) which measure-theoretically covers \(\M\), i.e.\ which satisfies:
\ew
\mu_\ps\lt( \M \m\osr \bigcup\mc{F}\rt) = 0,
\eew
where \(\mu_\ps\) denotes the Borel measure on \(\M\) induced by the Riemannian metric \(g_\ps\).
\end{Lem}

\begin{proof}
Let \(d_\ps\) denote the metric (i.e.\ distance function) on \(\M\) induced by the Riemannian metric \(g_\ps\) and write \(\ol{\mc{B}}_\eta(p)\) for the closed metric ball of radius \(\eta\) centred at \(p\).  By shrinking each \(\eta_p > 0\) if necessary, one may assume that \(\ol{\mc{B}}_\eta(p) = \ol{\B}_\eta(p)\) for all \(\eta \in (0,\eta_p]\) and thus:
\ew
\mc{G} = \lt\{ \ol{\mc{B}}_\eta(p) ~\m|~ p \in \M,\ \eta \in (0,\eta_p] \rt\}.
\eew

Initially, suppose that \(\M\) is compact (i.e.\ closed).  In this case, \(\M\) is a doubling metric measure space, in the sense that there exist constants \(C,R > 0\) such that for all \(p \in \M\) and \(\eta \in (0,R]\):
\e\label{doubling-met}
\mu_\ps\lt(\ol{\mc{B}}_{2\eta}(p)\rt) \le C\mu_\ps\lt(\ol{\mc{B}}_\eta(p)\rt).
\ee
Thus, by the Vitali Covering Theorem \cite[Thm.\ 1.6]{LoAoMS}, there exists a countable, disjoint measure-theoretic cover \(\mc{F} \pc \mc{G}\), as required.  (N.B.\ Whilst the statement of the Vitali Covering Theorem in \cite{LoAoMS} requires that \eref{doubling-met} hold for balls of arbitrarily large radius, the proof only requires the weaker condition stated above.)
 
Now suppose that \(\M\) is non-compact and fix a compact exhaustion \(\M_0 \pc \M_1 \pc \M_2 \pc ...\) of \(\M\) by compact manifolds with boundary.  (Such an exhaustion may be constructed by picking a proper function \(f: \M \to \bb{R}\), choosing increasing, unbounded sequences \(i_n,j_n \in (0,\0)\) such that \(i_n\) and \(-j_n\) are regular values of \(f\) for each \(n\) and setting \(\M_n = f^{-1}([-j_n,i_n])\).)  For each \(n \ge 1\), write \(\itr{\M}_n = \M_n\osr\del\M_n\) and consider the space \(\M'_n = \lt.\itr{\M}_n \m\osr \M_{n-1}\rt.\).  Then \(\lt(\M'_n, d_\ps|_{\M'_n}, \mu_\ps|_{\M'_n}\rt)\) is once again a doubling metric measure space (by compactness of \(\M_n\)) and thus one may choose a countable, disjoint, measure-theoretic cover \(\mc{F}_n\) of \(\M_n'\) from the set:
\ew
\mc{G}_n = \lt\{ \ol{\mc{B}}_\eta(p) ~\m|~ p \in \M'_n,\ \eta \in (0,\eta_p] \rt\} = \lt\{ \ol{\B}_\eta(p) ~\m|~ p \in \M'_n,\ \eta \in (0,\eta_p] \rt\},
\eew
where, by shrinking each \(\eta_p\) if necessary, one may assume that \(\ol{\B}_{\eta_p}(p) \pc \M'_n\) for each \(p \in \M'_n\).  The required cover of \(\M\) is then obtained by taking \(\mc{F} = \bigcup_{n \in \bb{B}} \mc{F}_n\).
 
\end{proof}

I now prove \tref{UB-Thm} in the \g\ case:

\begin{proof}
Choose \(\nu>0\) sufficiently small that:
\e\label{nuchoice}
(1+\ep)(1-\nu) \ge \lt(1+\frac{\ep}{2}\rt) \et 1 - \ep + \nu\ep &\le \lt(1-\frac{\ep}{2}\rt),
\ee
for \(\ep>0\) as defined in \lref{P}.  By applying \lref{cov-lem} to each connected component of the space \(\lt.\M \m\osr (\del\M \cup \M_{sing})\rt.\), one can obtain a countable, disjoint, measure-theoretic cover \(\lt\{\ol{\B_{\eta^{(i)}}(p^{(i)})}\rt\}_{i=0}^\infty\) of \(\M\) such that \(\ol{\B}_{\eta^{(i)}}\lt(p^{(i)}\rt) \pc \lt.\M \m\osr (\del\M \cup \M_{sing})\rt.\) for all \(i\) (note that \(\lt\{\ol{\B_{\eta^{(i)}}(p^{(i)})}\rt\}_{i=0}^\infty\) measure-theoretically covers all of \(\M\), rather than just \(\lt.\M \m\osr (\del\M \cup \M_{sing})\rt.\), as \(\del\M \cup \M_{sing}\) itself has zero measure in \(\M\)).  Then:
\ew
\CCH^\del(\ps) = \sum_{i=0}^\infty \CCH[,\ol{\B}_{\eta^{(i)}}\lt(p^{(i)}\rt)]^\del(\ps).
\eew
In particular, the right-hand sum is convergent and so there is \(N\ge 0\) such that:
\e\label{nubound}
\sum_{i=N+1}^\infty \CCH[,\ol{\B}_{\eta^{(i)}}\lt(p^{(i)}\rt)]^\del(\ps) < \nu\CCH^\del(\ps).
\ee
Now let:
\ew
\ps^\pm_1 = \ps + \sum_{i=0}^N \dd\be_{\eta^{(i)}}^{p^{(i)}/\pm},
\eew
where \(\be_{\eta^{(i)}}^{p^{(i)}/\pm}\) are defined according to \lref{local-ud}.  Then \(\ps_1^\pm \in [\ps]_+^\del\).  Moreover, by the estimates obtained in \lref{local-ud} and \erefs{nuchoice} and \eqref{nubound}:
\ew
\CCH^\del(\ps^-_1) &= \sum_{i=0}^N\CCH[,\ol{\B}_{\eta^{(i)}}\lt(p^{(i)}\rt)]^\del(\ps^-_1) + \sum_{i=N+1}^\infty\CCH[,\ol{\B}_{\eta^{(i)}}\lt(p^{(i)}\rt)]^\del(\ps^-_1)\\
&\le (1-\ep)\sum_{i=0}^N\CCH[,\ol{\B}_{\eta^{(i)}}\lt(p^{(i)}\rt)]^\del(\ps) + \sum_{i=N+1}^\infty\CCH[,\ol{\B}_{\eta^{(i)}}\lt(p^{(i)}\rt)]^\del(\ps)\\
&\le (1-\ep)\CCH^\del(\ps) + \ep\sum_{i=N+1}^\infty\CCH[,\ol{\B}_{\eta^{(i)}}\lt(p^{(i)}\rt)]^\del(\ps)\\
&\le (1-\ep+\nu\ep)\CCH^\del(\ps) \le \lt(1-\frac{\ep}{2}\rt)\CCH^\del(\ps)
\eew
and:
\ew
\CCH^\del(\ps^+_1) \ge \sum_{i=0}^N\CCH[,\ol{\B}_{\eta^{(i)}}\lt(p^{(i)}\rt)]^\del(\ps^+_1) \ge (1+\ep)(1-\nu)\CCH^\del(\ps) \ge \lt(1+\frac{\ep}{2}\rt)\CCH^\del(\ps).
\eew
Now recursively define \(\ps^\pm_n\) by applying the above argument to \(\ps^\pm_{n-1}\), respectively, for \(n\ge 2\). Then since the value of \(\ep\) in \lref{local-ud} is independent of the choice of closed \g\ 4-form, it follows that for all \(n\ge0\):
\ew
\CCH^\del(\ps^-_n) \le \lt(1 - \frac{\ep}{2}\rt)^n \CCH^\del(\ps) \to 0 \text{ as } n \to \infty
\eew
and:
\ew
\CCH^\del(\ps^+_n) \ge \lt(1 + \frac{\ep}{2}\rt)^n \CCH^\del(\ps) \to \infty \text{ as } n \to \infty,
\eew
as required.

\end{proof}

\section{The unboundedness above and below of \(\SCH^\del\) and \(\SCCH^\del\)}

The aim of this section is to prove \tref{UB-Thm} in the \sg\ case.  The key difference between the \sg\ case and the \g\ case is that \g-structures induce Riemannian metrics and thus the closed geodesic ball about a given point \(p\) defines a compact neighbourhood of \(p\), whilst \sg-structures induce pseudo-Riemannian metrics and thus induce closed geodesic balls about \(p\) which are neither compact nor neighbourhoods of \(p\).  To mitigate this difference, I make the following definition, based on \cite[Terminology 4.3]{A3F&ESLGoTG2}:
\begin{Defn}
Let \(\M\) be an oriented 7-manifold with \sg-structure and let \(\lt(\sph,\sps\rt)\) be the corresponding \sg\ 3- and 4-forms.  Given \(p \in \M\), an automorphism \(\mc{C} \in \GL_+(\T_p\M)\) is termed a Cartan involution of \(\T_p\M\) if \(\mc{C}^*\sph|_p = \sph|_p\) (equivalently \(\mc{C}^*\sps|_p = \sps|_p\)) and if the symmetric bilinear form \(g_{\sph}|_p(-,\mc{C}(-)) = g_{\sps}|_p(-,\mc{C}(-))\) is positive definite; denote this bilinear form by either \(h_{\sph,\mc{C}}\) or \(h_{\sps,\mc{C}}\), depending upon context.  A Cartan field on \(\M\) is then a section \(\mc{C}\) of the bundle \(\Aut(\T\M)\) such that \(\mc{C}|_p\) is a Cartan involution of \(\T_p\M\) for each \(p \in \M\).
\end{Defn}

As an example of the above definition, recall \(\svph_0 \in \ww{3}\lt(\bb{R}^7\rt)^*\) and \(\svps_0 \in \ww{4}\lt(\bb{R}^7\rt)^*\).  The automorphism:
\ew
\mc{C}_0 = \diag(+1,+1,+1,-1,-1,-1,-1) \in \tld{\Gg}_2
\eew
is a Cartan involution \wrt\ \(\lt(\svph_0,\svps_0\rt)\) and \(h_{\svph_0,\mc{C}_0} = h_{\svps_0,\mc{C}_0}\) is simply the Euclidean inner product on \(\bb{R}^7\).  More generally, it was proven in \cite[Prop.\ 4.4 \&\ Thm.\ 4.5]{A3F&ESLGoTG2} that \sg\ acts transitively on the space of Cartan involutions on \(\bb{R}^7\) and that the stabiliser of any Cartan involution on \(\bb{R}^7\) is isomorphic to the maximally compact subgroup \(\SO(4)\) of \sg, implying in particular that the space of Cartan involutions on \(\lt(\bb{R}^7,\svph_0,\svps_0\rt)\) is contractible.  Thus, given an oriented 7-manifold with \sg-structure \(\lt(\M,\sph,\sps\rt)\), there always exists a (homotopically unique) choice of Cartan field \(\mc{C}\) and moreover, for each \(p \in \M\), there exists an orientation-preserving isomorphism \(\al:\T_p\M \to \bb{R}^7\) such that:
\ew
\al^*\svph_0 = \sph|_p, \hs{3mm} \al^*\svps_0 = \sps|_p \et \al^{-1} \circ \mc{C}_0 \circ \al = \mc{C}|_p.
\eew

Given this additional terminology, one may now proceed as in the \g\ case.
\begin{Lem}\label{sg-analogue}
Let \(\sph\) be a closed \sg\ 3-form on \(\bb{R}^7\), let \(\mc{C}\) be a Cartan field and let \(\B_\eta\lt(\sph,\mc{C}\rt)\) denote the open geodesic ball of radius \(\eta\) centred at 0, defined by the metric \(h_{\sph,\mc{C}}\). Then there exist \(\eta_0 > 0\) (depending on \(\sph\) and \(\mc{C}\)) and \(\ep > 0\) (independent of \(\sph\) and \(\mc{C}\)) such that for all \(\eta \in (0,\eta_0]\):
\begin{enumerate}
\item There exist \(\tld{\al}^\pm_\eta \in \Om^2\lt(\ol{\B}_\eta\lt(\sph,\mc{C}\rt)\rt)\) with \(\supp(\tld{\al}^\pm_\eta) \C \B_\eta\lt(\sph,\mc{C}\rt)\) such that
\ew
\lt.\mc{D}^2\SCH[,\ol{\B}_\eta\lt(\sph,\mc{C}\rt)]^\del\rt|_{\sph}\lt(\dd\tld{\al}^\pm_\eta,\dd\tld{\al}^\pm_\eta\rt) \gl 0.
\eew

\item
There exist \(\tld{\be}^\pm_\eta\in\Om^2\lt(\ol{\B}_\eta\lt(\sph,\mc{C}\rt)\rt)\) with \(\supp\lt(\tld{\be}^\pm_\eta\rt)\C\B_\eta\lt(\sph,\mc{C}\rt)\) such that \(\sph + \dd\tld{\be}^\pm_\eta\) is a \sg\ 3-form and:
\ew
\SCH[,\ol{\B}_\eta\lt(\sph,\mc{C}\rt)]^\del\lt(\sph + \dd\tld{\be}^\pm_\eta\rt) \gle (1\pm\ep)\SCH[,\ol{\B}_\eta\lt(\sph,\mc{C}\rt)]^\del\lt(\sph\rt).
\eew
\end{enumerate}

Likewise, let \(\sps\) be a closed \sg\ 4-form on \(\bb{R}^7\) with Cartan field \(\mc{C}\).  Then the analogous conclusions apply.
\end{Lem}

\begin{proof}
Consider firstly the case of \sg\ 3-forms.  As in \sref{4Loc}, let:
\ew
\sph_0 = \dd x^{123} - \dd x^{145} - \dd x^{167} + \dd x^{246} - \dd x^{247} - \dd x^{347} - \dd x^{356} \et \mc{C}_0 = \diag(+1,+1,+1,-1,-1,-1,-1)
\eew
be the standard (torsion-free) \sg\ 3-form on \(\bb{R}^7\) (inducing the indefinite metric \(\tld{g}_0 = \sum_{i=1}^3 \lt(\dd x^i\rt)^{\ts2} - \sum_{i=4}^7 \lt(\dd x^i\rt)^{\ts2}\) and Euclidean volume form \(vol_0\)) and the standard Cartan field on \(\bb{R}^7\), respectively, and begin by considering (1) in the case \(\sph = \sph_0\), \(\mc{C} = \mc{C}_0\) and \(\eta = 1\).  Let \(r\) and \(f:[0,1]\to[0,1]\) be as in the proof of \lref{P0} and define:
\ew
\tld{\al}^+ = f(r)\dd x^{12} \et \tld{\al}^- = f(r)\dd x^{14}.
\eew
A direct calculation yields:
\caw
\lt\|\pi_1\lt(\dd\tld{\al}^+\rt)\rt\|_{\tld{g}_0} = \frac{1}{7}\lt(\frac{\dd f}{\dd r}\frac{1}{r}\rt)^2\lt(x^3\rt)^2, \hs{5mm} \lt\|\pi_7\lt(\dd\tld{\al}^+\rt)\rt\|_{\tld{g}_0} = -\frac{1}{4}\lt(\frac{\dd f}{\dd r}\frac{1}{r}\rt)^2\lt(\lt(x^4\rt)^2 + \lt(x^5\rt)^2 + \lt(x^6\rt)^2 + \lt(x^7\rt)^2\rt),\\
\lt\|\pi_{27}\lt(\dd\tld{\al}^+\rt)\rt\|_{\tld{g}_0} = \lt(\frac{\dd f}{\dd r}\frac{1}{r}\rt)^2\lt(\frac{6}{7}\lt(x^3\rt)^2 - \frac{3}{4}\lt(\lt(x^4\rt)^2 + \lt(x^5\rt)^2 + \lt(x^6\rt)^2 + \lt(x^7\rt)^2\rt)\rt),
\caaw
and thus, by \pref{Hitchin-derivs}:
\ew
\lt.\mc{D}^2\SCH[,\ol{\B}_1]^\del\rt|_{\sph_0}\lt(\dd\tld{\al}^+,\dd\tld{\al}^+\rt) = \bigintsss_{\ol{\B}_1} \frac{1}{3}\lt(\frac{\dd f}{\dd r}\frac{1}{r}\rt)^2\lt(-\frac{2\lt(x^3\rt)^2}{3} + \frac{\lt(x^4\rt)^2}{2} + \frac{\lt(x^5\rt)^2}{2} + \frac{\lt(x^6\rt)^2}{2} + \frac{\lt(x^7\rt)^2}{2}\rt) vol_0.
\eew
Likewise:
\caw
\lt\|\pi_1\lt(\dd\tld{\al}^-\rt)\rt\|_{\tld{g}_0} = \frac{1}{7}\lt(\frac{\dd f}{\dd r}\frac{1}{r}\rt)^2\lt(x^5\rt)^2, \hs{5mm} \lt\|\pi_7\lt(\dd\tld{\al}^-\rt)\rt\|_{\tld{g}_0} = \frac{1}{4}\lt(\frac{\dd f}{\dd r}\frac{1}{r}\rt)^2\lt(-\lt(x^2\rt)^2 - \lt(x^3\rt)^2 + \lt(x^6\rt)^2 + \lt(x^7\rt)^2\rt),\\
\lt\|\pi_{27}\lt(\dd\tld{\al}^-\rt)\rt\|_{\tld{g}_0} = \lt(\frac{\dd f}{\dd r}\frac{1}{r}\rt)^2\lt(\frac{6}{7}\lt(x^5\rt)^2 + \frac{3}{4}\lt(-\lt(x^2\rt)^2 - \lt(x^3\rt)^2 + \lt(x^6\rt)^2 + \lt(x^7\rt)^2\rt)\rt),
\caaw
and thus, by \pref{Hitchin-derivs}:
\ew
\lt.\mc{D}^2\SCH[,\ol{\B}_1]^\del\rt|_{\sph_0}\lt(\dd\tld{\al}^-,\dd\tld{\al}^-\rt) = \bigintsss_{\ol{\B}_1} \frac{1}{3}\lt(\frac{\dd f}{\dd r}\frac{1}{r}\rt)^2\lt(\frac{\lt(x^2\rt)^2}{2} + \frac{\lt(x^3\rt)^2}{2} - \frac{2\lt(x^5\rt)^2}{3} - \frac{\lt(x^6\rt)^2}{2} - \frac{\lt(x^7\rt)^2}{2}\rt) vol_0.
\eew
By symmetry, for any distinct \(i,j\in\{1,...,7\}\):
\ew
\bigintsss_{\ol{\B}_1}\lt(\frac{\dd f}{\dd r}\frac{1}{r}\rt)^2\lt(\lt(x^i\rt)^2 - \lt(x^j\rt)^2\rt) vol_0 = 0.
\eew
Thus, for \(\tld{\al}^+\), one sees that:
\ew
\lt.\mc{D}^2\SCH[,\ol{\B}_1]^\del\rt|_{\sph_0}\lt(\dd\tld{\al}^+,\dd\tld{\al}^+\rt) = \bigintsss_{\ol{\B}_1} \frac{1}{3}\lt(\frac{\dd f}{\dd r}\frac{1}{r}\rt)^2\lt(\frac{\lt(x^5\rt)^2}{3} + \frac{\lt(x^6\rt)^2}{2} + \frac{\lt(x^7\rt)^2}{2}\rt) vol_0 > 0.
\eew
Similarly, for \(\tld{\al}^-\), one sees that:
\ew
\lt.\mc{D}^2\SCH[,\ol{\B}_1]^\del\rt|_{\sph_0}\lt(\dd\tld{\al}^-,\dd\tld{\al}^-\rt) = -\bigintsss_{\ol{\B}_1} \frac{1}{3}\lt(\frac{\dd f}{\dd r}\frac{1}{r}\rt)^2\frac{2\lt(x^5\rt)^2}{3} vol_0 < 0.
\eew
Thus, (1) has been established in the case \(\sph = \sph_0\), \(\mc{C} = \mc{C}_0\) and \(\eta = 1\).  Moreover, (2) now follows in this case by taking \(\tld{\be}^\pm = t\tld{\al}^\pm\) for \(t > 0\) sufficiently small, as in \lref{P0}.

The general case now follows as in the proof of \lref{P} by noting firstly that, once again, both (1) and (2) are diffeomorphism invariant and thus that, \wlg, one may assume \(\sph|_0 = \sph_0|_0\) and \(\mc{C}|_0 = \mc{C}_0|_0\), and noting secondly that both (1) and (2) are invariant under the rescaling:
\ew
\eta \mt \la^\frac{1}{3}\eta, \hs{3mm} \sph \mt \la\sph, \hs{3mm} \tld{\al}^\pm_\eta \mt \la\tld{\al}^\pm_\eta  \et \tld{\be}^\pm_\eta \mt \la\tld{\be}^\pm_\eta
\eew
(with \(\mc{C}\) remaining constant) and that, writing \(\mu_\eta : x \in \bb{R}^7 \mt \eta x \in \bb{R}^7\) as in \sref{4Loc}, one has:
\ew
\eta^{-3}\mu_\eta^*\sph \to \sph_0 \text{ locally uniformly and } \eta^{-1}\ol{\B}_\eta\lt(\sph,\mc{C}\rt) \to \ol{\B}_1 \text{ in the Hausdorff sense as } \eta \to 0.
\eew\\

The argument for \sg\ 4-forms is entirely analogous.  For (1) in the case \(\sps = \sps_0\), \(\mc{C} = \mc{C}_0\) and \(\eta = 1\) one takes:
\ew
\tld{\al}^+ = f(r)\dd x^{123} \et \tld{\al}^- = f(r)\dd x^{124}
\eew
and computes that:
\ew
\lt.\mc{D}^2\SCCH[,\ol{\B}_1]^\del\rt|_{\sps_0}\lt(\dd\tld{\al}^+,\dd\tld{\al}^+\rt) = \bigintsss_{\ol{\B}_1} \frac{1}{4}\lt(\frac{\dd f}{\dd r}\frac{1}{r}\rt)^2\lt(\frac{\lt(x^4\rt)^2}{2} + \frac{\lt(x^5\rt)^2}{2} + \frac{\lt(x^6\rt)^2}{2} + \frac{\lt(x^7\rt)^2}{2}\rt) vol_0 > 0
\eew
and:
\ew
\lt.\mc{D}^2\SCCH[,\ol{\B}_1]^\del\rt|_{\sps_0}\lt(\dd\tld{\al}^-,\dd\tld{\al}^-\rt) &= \bigintsss_{\ol{\B}_1} \frac{1}{4}\lt(\frac{\dd f}{\dd r}\frac{1}{r}\rt)^2\lt(\frac{\lt(x^3\rt)^2}{2} - \frac{\lt(x^5\rt)^2}{2} - \frac{\lt(x^6\rt)^2}{2} - \frac{3\lt(x^7\rt)^2}{4}\rt) vol_0\\
&= -\bigintsss_{\ol{\B}_1}\lt(\frac{\dd f}{\dd r}\frac{1}{r}\rt)^2\lt(\frac{\lt(x^6\rt)^2}{2} + \frac{3\lt(x^7\rt)^2}{4}\rt) vol_0 < 0,
\eew
using symmetry, as before.  The rest of the argument now proceeds as above.

\end{proof}

Given an arbitrary compact, oriented 7-manifold \(\M\) with boundary equipped with a closed \sg\ 3-form \(\sph\), one may now use \lref{sg-analogue} together with the argument in \sref{GMT} to construct \(\sph_n^\pm \in \lt[\sph\rt]_{\tl}^\del\) such that:
\ew
\lim_{n \to \0} \SCH^\del\lt(\sph_n^+\rt) \to \0 \et \lim_{n \to \0} \SCH^\del\lt(\sph_n^-\rt) \to 0,
\eew
where, at each step, one first chooses a Cartan field \(\mc{C}_n^\pm\) (\wrt\ \(\sph_n^\pm\)) on each component of the set \(\lt.\M\m\osr(\del\M \cup \M_{sing})\rt.\) before constructing \(\sph_{n+1}^\pm\).  The argument in the 4-form case is identical.  Thus, it has been proven that:\vs{3mm}

\noindent{\bf Theorem \ref{UB-Thm}} (\sg\ case){\bf .}
\em The functionals \(\SCH^\del\) and \(\SCCH^\del\) are always unbounded above and below.  Specifically, let \(\M\) be a compact, oriented 7-orbifold with boundary and let \(\sph\) be a closed \sg\ 3-form on \(\M\).  Then:
\ew
\inf_{\sph' \in \lt[\sph\rt]^\del_{\tl}} \SCH^\del\lt(\sph'\rt) = 0 \et \sup_{\sph' \in \lt[\sph\rt]^\del_{\tl}} \SCH^\del\lt(\sph'\rt) = \infty.
\eew
Likewise, the analogous statement holds for the functional \(\SCCH^\del\).  In particular, by taking \(\del\M = \es\), the same conclusions apply to the Hitchin functionals \(\SCH\) and \(\SCCH\).\vs{2mm}\em

\qed\vs{3mm}

\section{The unboundedness below of \(\CH^\del\)}\label{CHInput}

The aim of this section is to prove \tref{MT'}.  Clearly the deduction of \lref{P0}(2) from \lref{P0}(1), the proof of \lref{P} and the arguments of \sref{GMT} all apply equally to either \(\CH^\del\) or \(\CCH^\del\).  However, \lref{P0}(1) has no exact analogue for \(\CH^\del\).  Indeed, by \cite[Prop.\ 1]{BVPiD74&3RtEH}, writing \(\B_1\pc\bb{R}^7\) for the open unit ball in \(\bb{R}^7\) and:
\ew
\ph_0 = \dd x^{123} + \dd x^{145} + \dd x^{167} + \dd x^{246} - \dd x^{247} - \dd x^{347} - \dd x^{356}
\eew
for the standard (torsion-free) \g\ 3-form on \(\bb{R}^7\), the Hessian \(\mc{D}^2\CH[,\ol{\B}_1]^\del\) is non-positive definite (note that one may embed \(\lt(\ol{\B}_1,\ph_0\rt) \emb \lt(\rqt{\bb{R}^7}{3\bb{Z}^7},\ph_0\rt)\)).  There is, however, a partial analogue of \lref{P0} for \(\CH^\del\):
\begin{Lem}\label{CH-analogue}
There exists \(\al^-\in\Om^2\lt(\ol{\B}_1\rt)\) with \(\supp(\al^-)\C\B_1\) such that:
\ew
\lt.\mc{D}^2\CH[,\ol{\B}_1]^\del\rt|_{\ph_0}(\dd\al^-,\dd\al^-) < 0.
\eew
\end{Lem}

\begin{proof}
Consider \(\al^- = f(r)\cdot \dd x^{12}\), for \(f\) and \(r\) as defined in the proof of \lref{P0}.  A direct calculation yields that:
\ew
\lt.\mc{D}^2\CH[,\ol{\B}_1]^\del\rt|_{\ph_0}(\dd\al^-,\dd\al^-) = -\bigintsss_{\ol{\B}_1} \frac{1}{3}\lt(\frac{\dd f}{\dd r}\frac{1}{r}\rt)^2\lt(\frac{2}{3}(x^3)^2 + \frac{1}{2}\lt((x^4)^2 + (x^5)^2 + (x^6)^2 + (x^7)^2\rt)\rt) vol_0 < 0,
\eew
as required.
\end{proof}

By \lref{CH-analogue}, \tref{MT'} follows at once:\vs{3mm}

\noindent{\bf Theorem \ref{MT'}.}
\em The functional \(\CH^\del\) is always unbounded below.  Specifically, let \(\M\) be a compact, oriented 7-orbifold with boundary, let \(\ph\) be a closed \g\ 3-form on \(\M\).  Then:
\ew
\inf_{\ph' \in [\ph]^\del_+} \CH^\del\lt(\ph'\rt) = 0.
\eew
In particular, by taking \(\del\M = \es\), the same conclusion applies to the Hitchin functional \(\CH\).\vs{2mm}\em

\qed

\section{Applications: classification of the critical points of \(\CCH^\del\), \(\SCH^\del\) and \(\SCCH^\del\), and density of initial conditions of Laplacian coflow which do not lead to convergent solutions}

The aim of this section is to prove \trefs{saddle-thm} and \ref{donco-thm}.\vs{3mm}

\noindent{\bf Theorem \ref{saddle-thm}.}
\em The critical points of \(\CCH^\del\), \(\SCH^\del\) and \(\SCCH^\del\) are always saddles.  Specifically, let \(\M\) be a compact, oriented 7-orbifold with (possibly empty) boundary and let \(\ps\) be a torsion-free \g\ 4-form on \(\M\).  Then there exist infinite-dimensional subspaces \(\mc{S}_4^\pm\lt(\ps\rt) \pc \T_\ps [\ps]_+^\del\) along which \(\lt.\mc{D}^2\CCH^\del\rt|_\ps\) is positive definite and negative definite, respectively.  The analogous statement holds for the functionals \(\SCH^\del\) and \(\SCCH^\del\).\vs{2mm}\em

\begin{proof}
Begin with the case of \g\ 4-forms.  Let \(\lt\{U_i\rt\}_{i \in \bb{N}}\) be a countable disjoint collection of open subsets of \(\M\osr(\del\M \cup \M_{sing})\).  By \lref{P}, for each \(i \in \bb{N}\) there exist 3-forms \(\al_i^\pm\) on \(\M\) with \(\supp(\al_i^\pm) \C U_i\) such that \(\lt.\mc{D}^2\CCH^\del\rt|_\ps\) is positive definite (respectively, negative definite) along \(\dd\al_i^\pm\).  Now take:
\ew
\mc{S}^\pm_4(\ps) = \Ds_{i \in \bb{N}} \bb{R} \1 \dd\al^\pm_i \cc \dd\Om^3(\M;\del\M).
\eew
One can verify that \(\lt.\mc{D}^2\CCH^\del\rt|_\ps\) is positive and negative definite along \(\mc{S}^\pm_4(\ps)\) respectively, as required.  The proof in the case of \sg\ 3- or 4-forms is analogous.

\end{proof}

Now turn to \tref{donco-thm}.  Let \(\M\) be a (possibly non-compact) oriented 7-manifold without boundary, let \(\ps \in \Om^4_+(\M)\) be a closed \g\ 4-form and recall the set \([\ps]_+^\del = [\ps]_+ \pc \Om^4(\M)\) defined in \eref{be+-defn}.  Given a Riemannian metric \(g\) on \(\M\) and a countable exhaustion of \(\M\) by compact subsets \(K_0 \cc K_1 \cc ... \cc \M\),  the countable family of seminorms \(\|-\|_{C^0_g(K_n)}\) on \(\Om^4(\M)\) is separating and induces the \(C^0\) topology on \(\Om^4(\M)\) (and hence on \([\ps]_+\)); this topology is independent of the choice of \(g\) and \(K_n\).

Next, recall that the Laplacian coflow of \(\ps\) is the solution of the evolution PDE:
\ew
\frac{\del \ps(t)}{\del t} = \De_{\ps(t)}\ps(t) = \dd\dd^*_{\ps(t)}\ps(t) \et \ps(0) = \ps.
\eew
Using this terminology, I now prove \tref{donco-thm}.  Recall the statement of the theorem:\vs{3mm}

\noindent{\bf Theorem \ref{donco-thm}.}
\em Let \(\M\) be a (possibly non-compact) oriented 7-manifold without boundary and let \(\ps \in \Om^4_+(\M)\) be a closed \g\ 4-form.  Consider the space:\vs{-2mm}
\ew
\mc{O}_{[\ps]_+} = \bigg\{\ps'\in[\ps]_+~\bigg| \parbox{7.5cm}{\begin{center}no solution to the Laplacian coflow started at \(\ps'\) converges to a torsion-free \g\ 4-form\end{center}}\bigg\}.
\eew
Then \(\mc{O}_{[\ps]_+} \pc [\ps]_+\) is dense in the \(C^0\) topology.\vs{2mm}\em

\begin{proof}
Begin by considering the ball \(\ol{\B}_\eta \pc \bb{R}^7\) equipped with the standard flat \g\ 4-form \(\ps_0\) and choose \(\al^+_\eta \in \Om^3\lt(\ol{\B}_\eta\rt)\) with \(\supp\lt(\al^+_\eta\rt) \C \ol{\B}_\eta\) such that:
\e\label{LCF-input}
\lt.\mc{D}^2\CCH[,\ol{\B}_\eta]^\del\rt|_{\ps_0}(\dd\al^+_\eta,\dd\al^+_\eta) > 0,
\ee
according to \lref{P}.  I begin by proving that for all \(s > 0\) sufficiently small, the Laplacian coflow on \(\B_\eta\) starting from \(\ps_0 + s\dd\al^+_\eta\) cannot converge to a torsion-free \g\ 4-form.

Indeed, let \(\h{\ps}\) be a torsion-free \g\ 4-form on \(\ol{\B}_\eta\) such that:
\e\label{CS}
\supp\big(\h{\ps}-\ps_0\big)\C\B_\eta.
\ee
Since \(\h{\ps}\) is torsion-free, \(g_{\h{\ps}}\) is Ricci-flat \cite[Prop.\ 11.8]{RG&HG}.  Moreover, the mean curvature of \(\del\ol{\B_\eta} = S^6_\eta\) as a submanifold of \((\ol{\B}_\eta,g_{\h{\ps}})\) \wrt\ the inwards pointing normal is \(\frac{1}{\eta}\), since \(g_{\h{\ps}}\) is simply the Euclidean metric in a neighbourhood of \(S^6_\eta\), by \eref{CS}. Thus, using \cite[Thm.\ 2.1]{AGCTwAtVEfS}, it follows that:
\e\label{HK-vol-bd}
\CCH[,\ol{\B}_\eta]^\del\lt(\h{\ps}\rt) &\le \bigintsss_0^\eta \bigg(1-\frac{r}{\eta}\bigg)^7\dd r \cdot \text{Vol}_{S^6_\eta}\big(\h{\ps}\big)\\
& = \frac{\eta}{7}\text{Vol}_{S^6_\eta}(\ps_0),
\ee
where \(\text{Vol}_{S^6_\eta}\big(\h{\ps}\big)\) is the volume of \(S^6_\eta\) \wrt\ the metric on \(S^6_\eta\) induced by \(\h{\ps}\), which is the same as the metric induced by \(\ps_0\), using \eref{CS}.  A direct calculation shows that \eref{HK-vol-bd} is saturated when \(\h{\ps} = \ps_0\). Hence, for all torsion-free \(\h{\ps}\) satisfying \eref{CS}:
\e\label{TFVB}
\CCH[,\ol{\B}_\eta]^\del\big(\h{\ps}\big) \le \CCH[,\ol{\B}_\eta]^\del(\ps_0).
\ee

Now let \(\ps_s(t)\) denote a solution to the Laplacian coflow started at \(\ps_0  + s\dd\al^+_\eta\) and suppose that \(\ps_s(t)\) existed for all \(t\) and converged to a torsion-free \g\ 4-form \(\h{\ps}\) as \(t \to \infty\).  Since Laplacian coflow preserves \(\ps_0\), \(\ps_s(t)\) is fixed on the region where \(\ps_s(0) = \ps_0\) and hence:
\ew
\supp\lt(\h{\ps}- \ps_0\rt) \cc \supp(\ps_s(0) - \ps_0) \C \B_\eta.
\eew
(In particular, each \(\ps_s(t)\) can be extended to the boundary of \(\ol{\B}_\eta\).)  Thus, using \eref{TFVB}, one has \(\CCH[,\ol{\B}_\eta]^\del\lt(\h{\ps}\rt) \le \CCH[,\ol{\B}_\eta]^\del(\ps_0)\).  However, the Laplacian coflow increases volume pointwise \cite[eqn.\ (4.32)]{STBoaMLCoG2S}. Hence:
\ew
\CCH[,\ol{\B}_\eta]^\del\lt(\h{\ps}\rt) &\ge \CCH[,\ol{\B}_\eta]^\del(\ps_s(0)) = \CCH[,\ol{\B}_\eta]^\del(\ps_0 + s\dd\al_\eta^+)\\
&> \CCH[,\ol{\B}_\eta]^\del(\ps_0)\\
\eew
where the last line follows from \eref{LCF-input} for all \(s > 0\) sufficiently small (cf.\ the proof of \lref{P0}), contradicting \eref{TFVB}.

Thus, there are (uniformly) arbitrarily small, compactly supported perturbations \(\ps_s(0) = \ps_0 + s\dd\al^+_\eta\) of \(\ps_0\) such that the Laplacian coflow on \(\B_\eta\) started from \(\ps_s(0)\) cannot converge to a torsion-free \g\ 4-form.  To complete the proof, therefore, it suffices to prove that, given any \((\M,\ps)\) as in the statement of the theorem and any \(\ps'\in[\ps]_+\), there exists a closed \g\ 4-form \(\ps'' \in [\ps]_+\), arbitrarily close to \(\ps'\) in the \(C^0\) topology, such that \(\ps''\) is diffeomorphic to \(\ps_0\) in some small neighbourhood of \(\M\).  This follows from the subsequent local result:
\begin{Cl}
Let \(\ps'\) be a closed \g\ 4-form on \(\bb{R}^7\) such that \(\ps|_0 = \ps_0|_0\).  Then for all \(\de > 0\), there exists \(\eta \in (0,1]\) and \(\al \in \Om^3(\bb{R}^7)\) with \(\supp(\al) \C \B_{2\eta}\) such that:
\ew
\ps' + \dd\al = \ps_0 \text{ on } \B_\eta \et \|\dd\al\|_{C^0_{\ps'}} < \de.
\eew
(Note that \(\ps' + \dd\al\) is automatically of \g-type for \(\de>0\) sufficiently small, by the stability of \g\ 4-forms.)
\end{Cl}

\begin{proof}[Proof of Claim]
Consider the 4-form \(\ps_0-\ps'\). Since \(\ps_0-\ps'\) vanishes at 0, there is some constant \(C_1>0\) such that for all \(\eta \in (0,1]\):
\e\label{linbd}
\|\ps_0-\ps'\|_{C^0_{\ps'}(\B_{2\eta})} \le C_1\eta.
\ee
Similarly, since \(\dd(\ps_0-\ps')=0\), one can choose a primitive \(\vpi\in\Om^3\lt(\bb{R}^7\rt)\) for \(\ps_0-\ps'\) such that, for some constant \(C_2>0\) and all \(\eta \in (0,1]\):
\e\label{quadrbd}
\|\vpi\|_{C^0_{\ps'}(\B_{2\eta})} \le C_2\eta^2.
\ee
Indeed (cf.\ \cite[p.\ 16]{ACG2CMwFBNb1eq1}) identify \(\lt.\bb{R}^7\rt\osr\{0\} \cong (0,\infty) \x S^6\) and write \(\ps_0 - \ps' = \si_1 + \dd t \w \si_2\), where \(t\) is the parameter along \((0,\infty)\), \(\si_i\) depends parametrically on \(t\) (\(i = 1,2\)), \(\dd \si_1 = 0\) and \(\frac{\del \si_1}{\del t} = \dd \si_2\) (since \(\dd(\ps_0 - \ps') = 0\)).  Set \(\vpi = \bigintsss_0^t  \si_2 \dd t\).  Then:
\ew
\dd\vpi = \bigintsss_0^t \frac{\del \si_1}{\del t} \dd t + \dd t \w \si_2 = \si_1 + \dd t \w \si_2 = \ps_0 - \ps'
\eew
(since \(\si_1(t) \to 0\) as \(t \to 0\)) and \(\vpi\) clearly satisfies \eref{quadrbd}, as required.

Next, fix a smooth function \(f:[0,2]\to[0,1]\) such that \(f\equiv 1\) on \([0,1]\) and \(f \equiv 0\) on a neighbourhood of \(2\). Given any \(\eta \in (0,1]\), define \(f_\eta(r) = f\lt(\frac{r}{\eta}\rt)\) and set \(\al = f_\eta(r)\vpi\).  Then \(\ps' + \dd\al = \ps' + \dd\vpi = \ps_0\) on \(\B_\eta\), as required.  Moreover:
\ew
\|\dd\al\|_{C^0_{\ps'}} &= \lt\|\frac{\dd f_\eta}{\dd r}\dd r\w\vpi + f_\eta\dd\vpi\rt\|_{C^0_{\ps'}}\\
&\le \frac{\sup |f'|}{\eta} \|\dd r\w\vpi\|_{C^0_{\ps'}(\B_{2\eta})} + \|\ps_0 - \ps'\|_{C^0_{\ps'}(\B_{2\eta})}\\
&\le C_3\eta,
\eew
where the first inequality follows from the fact that \(\supp(f_\eta) \C \B_{2\eta}\) and the second inequality follows from \erefs{linbd} and \eqref{quadrbd}.  The claim now follows by taking \(\eta \in (0,1]\) sufficiently small. This in turn completes the proof of \tref{donco-thm}.

\let\qed\relax
\end{proof}

\end{proof}

~\vs{5mm}

\noindent Laurence H.\ Mayther\\
University of Cambridge\\
United Kingdom\\
{\it lhm32@cam.ac.uk}

\end{document}